\documentclass[11pt]{amsart}

\usepackage[T1]{fontenc}
\usepackage{lmodern}

\usepackage{amsmath,amsfonts,amssymb,amsthm}
\usepackage{graphics,subfig}
\usepackage{esint}
\usepackage{shuffle}
\usepackage{istgame}
\usepackage{footmisc}
\usepackage{tikz}
\usepackage{circuitikz}

\usepackage{lineno}

\usepackage{enumerate}
\usepackage{aligned-overset}

\newcommand{ \E }{ \mathbb{E} }
\newcommand{\PP}{\mathbb{P} }

\newcommand{\ignore}[1]{}

\usepackage[hidelinks,hyperindex,breaklinks]{hyperref}

\newcommand{\len}{\text{Len}}
\newcommand{\Exc}{\text{Exc}}

\newtheorem{theorem}{Theorem}
\newtheorem{remark}{Remark}
\newtheorem{lemma}{Lemma}
\newtheorem{corollary}{Corollary}

\title[An isoperimetric problem with a white-noise volume term]{Asymptotics of a planar isoperimetric problem with a white-noise volume term}

\author{Xiaopeng Cheng, Felix Otto, Matteo Palmieri}

\begin{document}

\begin{abstract}
In this work we study the maximal ratio $I$ between the white noise integrated over a set and the perimeter of the set, which can be seen as a random isoperimetric problem, and appears in the random-field Ising model (Ding-Wirth) and min-max optimal matching (Leighton-Shor). In the planar case considered here, such a ratio is scale-invariant and therefore the problem is critical. As such it requires an ultraviolet cutoff, which we impose by restricting to polygonal sets with side-length at least 1 contained in the ball $B_L$. Our main result establishes the leading-order asymptotics $\mathbb{E}I\approx i\ln^{3/4}L$ for some $i\in(0,\infty)$ and superconcentration at scale $O(\ln^{-1/4}L)$. This is done by relating to a simpler $(1+1)$-dimensional action, studied by the last two authors and C.~Wagner. Such a connection is found by a classical geometric linearization of the perimeter to the Dirichlet energy, justified by Ried-Wagner large-scale regularity theory, and allows a coarse-graining and scale-by-scale iteration argument.
\end{abstract}

\maketitle

\tableofcontents

\section{Introduction}\label{S:introduction}

\subsection{The random isoperimetric problem and its asymptotics}

Given a closed, oriented, curve $\gamma$ in the plane $\mathbb{R}^2$, we denote by $\len(\gamma)$ its length and, for every point $z$ not on the curve, we write $w(\gamma,z)\in\mathbb{Z}$ for its winding number, namely the number of times $\gamma$ goes around\footnote{by convention $w(\gamma,z)$ is positive counter-clockwise and negative clockwise. See the definition in \cite[Subsection 5.1]{dwarxiv}} $z$. For a two-dimensional white noise $\xi$, let us introduce the field term
\begin{align}\label{m107}
W(\gamma):=\int_{\mathbb{R}^2}w(\gamma,z)\xi(z).
\end{align}
This work is devoted to the study of the variational problem
\begin{align*}
\max_\gamma\frac{W(\gamma)}{\len(\gamma)}.
\end{align*}

\medskip

As it is stated, the problem is ill-posed, namely the maximal ratio is naively infinite. Indeed, by picking all the possible integer translates of a fixed circle of radius $\textstyle{\frac{1}{2}}$, one obtains infinitely many independent competitors with the same law, in view of the spatial independence of the noise. An easy fix is then to assume that the curves are contained in a certain compact set. There is however a second, more severe, obstruction, given by criticality: The two functionals scale linearly under isotropic rescalings,
\begin{align}\label{m104}
\gamma=t\hat\gamma\quad\mbox{implies}\quad \len(\gamma)=t\len(\hat\gamma)\quad\mbox{and}\quad W(\gamma)=_{\text{law}}t W(\hat\gamma),
\end{align}
and thus the law of $W(\gamma)/\len(\gamma)$ is the same for every circle $\gamma$, independently of its radius. Hence the ill-posedness persists, since every open bounded domain can be filled with infinitely many disjoint circles with vanishing radii.

\medskip

In order to restore well-posedness, one needs to fix both a macroscopic scale $L$ and a microscopic scale\footnote{These are also called infrared and ultraviolet cutoffs, respectively.}, which in view of the scaling \eqref{m104} can be set to 1. We do that by restricting to the configuration space
\begin{align*}
\Gamma_{L,1}:=\{&\mbox{closed, piecewise-linear curves $\gamma$}\\&\mbox{with image $\subset B_L$ and side-length $\ge1$}\},
\end{align*}
where $B_L$ is the ball with radius $L$ and center at 0. As is common in other critical settings, the problem should be oblivious to the precise small-scale regularization and to the macroscopic shape of the domain. Regarding the latter, we refer to Subsection~\ref{ss:relative} below. While one could think of many other ways of introducing the ultraviolet cutoff, the way we do it, which is inspired by \cite{dw}, preserves the isotropy of the problem and fits nicely with the coarse-graining procedure explained below.

\medskip

Our quantity of interest is therefore
\begin{align}\label{m79}
I_L:=\max_{\gamma\in\Gamma_{L,1}}\frac{W(\gamma)}{\len(\gamma)}.
\end{align}
The main result of this work provides the leading-order asymptotic of its expected value.

\begin{theorem}\label{T:1}
There exists
\begin{align*}
\lim_{L\uparrow\infty}\frac{\mathbb{E}I_L}{\ln^{3/4}L}\in(0,\infty).
\end{align*}
\end{theorem}

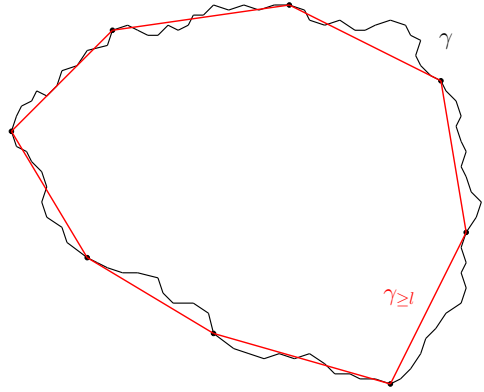
\begin{figure}
\centering
\resizebox{0.5\textwidth}{!}{%
\begin{circuitikz}[scale=0.5]
\tikzstyle{every node}=[font=\LARGE]
\node at (-75,47.75) [circ] {};
\draw [ line width=0.8pt](-75,47.75) to[short] (-74.75,48.5);
\draw [ line width=0.8pt](-74.75,48.5) to[short] (-74.5,49);
\draw [ line width=0.8pt](-74.5,49) to[short] (-74,49.25);
\draw [ line width=0.8pt](-74,49.25) to[short] (-73.75,49.75);
\draw [ line width=0.8pt](-73.75,49.75) to[short] (-73.25,49.5);
\draw [ line width=0.8pt](-73.25,49.5) to[short] (-72.75,50);
\draw [ line width=0.8pt](-72.75,50) to[short] (-72.5,50.75);
\draw [ line width=0.8pt](-72.5,50.75) to[short] (-71.75,51);
\draw [ line width=0.8pt](-71.75,51) to[short] (-71.25,51.75);
\draw [ line width=0.8pt](-71.25,51.75) to[short] (-70.25,52);
\draw [ line width=0.8pt](-70.25,52) to[short] (-70,52.75);
\node at (-70,52.75) [circ] {};
\draw [ line width=0.8pt](-70,52.75) to[short] (-69.25,53);
\draw [ line width=0.8pt](-69.25,53) to[short] (-68.75,52.75);
\draw [ line width=0.8pt](-68.75,52.75) to[short] (-68.25,52.5);
\draw [ line width=0.8pt](-68.25,52.5) to[short] (-67.75,52.5);
\draw [ line width=0.8pt](-67.75,52.5) to[short] (-67.25,53);
\draw [ line width=0.8pt](-67.25,53) to[short] (-66.75,52.75);
\draw [ line width=0.8pt](-66.75,52.75) to[short] (-66.25,53);
\draw [ line width=0.8pt](-66.25,53) to[short] (-66,53.5);
\draw [ line width=0.8pt](-66,53.5) to[short] (-65.5,53.5);
\draw [ line width=0.8pt](-65.5,53.5) to[short] (-65,54);
\draw [ line width=0.8pt](-65,54) to[short] (-64.25,53.75);
\draw [ line width=0.8pt](-64.25,53.75) to[short] (-63.5,54);
\draw [ line width=0.8pt](-63.5,54) to[short] (-62.75,53.75);
\draw [ line width=0.8pt](-62.75,53.75) to[short] (-62,54);
\draw [ line width=0.8pt](-62,54) to[short] (-61.25,54);
\node at (-61.25,54) [circ] {};
\draw [ line width=0.8pt](-61.25,54) to[short] (-60.5,53.75);
\draw [ line width=0.8pt](-60.5,53.75) to[short] (-60.25,53.25);
\draw [ line width=0.8pt](-60.25,53.25) to[short] (-59.5,53);
\draw [ line width=0.8pt](-59.5,53) to[short] (-58.75,53.25);
\draw [ line width=0.8pt](-58.75,53.25) to[short] (-58,52.75);
\draw [ line width=0.8pt](-58,52.75) to[short] (-57.5,52.75);
\draw [ line width=0.8pt](-57.5,52.75) to[short] (-57,53);
\draw [ line width=0.8pt](-57,53) to[short] (-56.25,53.25);
\draw [ line width=0.8pt](-56.25,53.25) to[short] (-55.5,53);
\draw [ line width=0.8pt](-55.5,53) to[short] (-55.25,52.5);
\draw [ line width=0.8pt](-55.25,52.5) to[short] (-54.75,52.25);
\draw [ line width=0.8pt](-54.75,52.25) to[short] (-55,51.5);
\draw [ line width=0.8pt](-55,51.5) to[short] (-54.75,51);
\draw [ line width=0.8pt](-54.75,51) to[short] (-54.25,50.5);
\draw [ line width=0.8pt](-54.25,50.5) to[short] (-53.75,50.25);
\node at (-53.75,50.25) [circ] {};
\node [font=\LARGE] at (-53.5,52.2) {$\gamma$};
\draw [ line width=0.8pt](-53.75,50.25) to[short] (-53.5,49.75);
\draw [ line width=0.8pt](-53.5,49.75) to[short] (-53,49.25);
\draw [ line width=0.8pt](-53,49.25) to[short] (-52.75,48.5);
\draw [ line width=0.8pt](-52.75,48.5) to[short] (-53,48);
\draw [ line width=0.8pt](-53,48) to[short] (-52.75,47.25);
\draw [ line width=0.8pt](-52.75,47.25) to[short] (-52.5,46.75);
\draw [ line width=0.8pt](-52.5,46.75) to[short] (-52.75,46);
\draw [ line width=0.8pt](-52.75,46) to[short] (-52.5,45.5);
\draw [ line width=0.8pt](-52.5,45.5) to[short] (-52.25,44.75);
\draw [ line width=0.8pt](-52.25,44.75) to[short] (-51.75,44.25);
\draw [ line width=0.8pt](-51.75,44.25) to[short] (-52,43.5);
\draw [ line width=0.8pt](-52,43.5) to[short] (-52.5,42.75);
\node at (-52.5,42.75) [circ] {};
\draw [ line width=0.8pt](-52.5,42.75) to[short] (-52.75,42);
\draw [ line width=0.8pt](-52.75,42) to[short] (-52.5,41.25);
\draw [ line width=0.8pt](-52.5,41.25) to[short] (-52.75,40.75);
\draw [ line width=0.8pt](-52.75,40.75) to[short] (-52.5,40);
\draw [ line width=0.8pt](-52.5,40) to[short] (-52.75,39.25);
\draw [ line width=0.8pt](-52.75,39.25) to[short] (-53.5,38.75);
\draw [ line width=0.8pt](-53.5,38.75) to[short] (-54,38);
\draw [ line width=0.8pt](-54,38) to[short] (-54.5,37.5);
\draw [ line width=0.8pt](-54.5,37.5) to[short] (-54.75,36.75);
\draw [ line width=0.8pt](-54.75,36.75) to[short] (-55.25,36);
\draw [ line width=0.8pt](-55.25,36) to[short] (-55.75,35.5);
\draw [ line width=0.8pt](-55.75,35.5) to[short] (-56.25,35.25);
\node at (-56.25,35.25) [circ] {};
\draw [ line width=0.8pt](-56.25,35.25) to[short] (-57,35.5);
\draw [ line width=0.8pt](-57,35.5) to[short] (-57.75,35.5);
\draw [ line width=0.8pt](-57.75,35.5) to[short] (-58.25,35.75);
\draw [ line width=0.8pt](-58.25,35.75) to[short] (-59,35.75);
\draw [ line width=0.8pt](-59,35.75) to[short] (-59.5,36.25);
\draw [ line width=0.8pt](-59.5,36.25) to[short] (-60.25,36.75);
\draw [ line width=0.8pt](-60.25,36.75) to[short] (-61,36.5);
\draw [ line width=0.8pt](-61,36.5) to[short] (-61.75,36.75);
\draw [ line width=0.8pt](-61.75,36.75) to[short] (-62.5,36.5);
\draw [ line width=0.8pt](-62.5,36.5) to[short] (-63.25,36.75);
\draw [ line width=0.8pt](-63.25,36.75) to[short] (-64,37);
\draw [ line width=0.8pt](-64,37) to[short] (-65,37.75);
\node at (-65,37.75) [circ] {};
\draw [ line width=0.8pt](-65,37.75) to[short] (-65.25,38.75);
\draw [ line width=0.8pt](-65.25,38.75) to[short] (-66,39.25);
\draw [ line width=0.8pt](-66,39.25) to[short] (-67,39.25);
\draw [ line width=0.8pt](-67,39.25) to[short] (-67.5,39.75);
\draw [ line width=0.8pt](-67.5,39.75) to[short] (-67.75,40.5);
\draw [ line width=0.8pt](-67.75,40.5) to[short] (-68.75,40.75);
\draw [ line width=0.8pt](-68.75,40.75) to[short] (-69.5,40.75);
\draw [ line width=0.8pt](-69.5,40.75) to[short] (-70.25,41);
\draw [ line width=0.8pt](-70.25,41) to[short] (-70.75,41.25);
\draw [ line width=0.8pt](-70.75,41.25) to[short] (-71.25,41.5);
\draw [ line width=0.8pt](-75,47.75) to[short] (-74.75,47);
\draw [ line width=0.8pt](-74.75,47) to[short] (-74.25,46.75);
\draw [ line width=0.8pt](-74.25,46.75) to[short] (-74,46.25);
\draw [ line width=0.8pt](-74,46.25) to[short] (-73.5,45.75);
\draw [ line width=0.8pt](-73.5,45.75) to[short] (-73.25,45);
\draw [ line width=0.8pt](-73.25,45) to[short] (-73.5,44.25);
\draw [ line width=0.8pt](-73.5,44.25) to[short] (-73.25,43.5);
\draw [ line width=0.8pt](-73.25,43.5) to[short] (-72.5,43);
\draw [ line width=0.8pt](-72.5,43) to[short] (-72.25,42.25);
\draw [ line width=0.8pt](-72.25,42.25) to[short] (-71.25,41.5);
\node at (-71.25,41.5) [circ] {};
\draw [ color={rgb,255:red,255; green,0; blue,0}, , line width=1pt](-75,47.75) to[short] (-71.25,41.5);
\draw [ color={rgb,255:red,255; green,0; blue,0}, , line width=1pt](-71.25,41.5) to[short] (-65,37.75);
\draw [ color={rgb,255:red,255; green,0; blue,0}, , line width=1pt](-65,37.75) to[short] (-56.25,35.25);
\draw [ color={rgb,255:red,255; green,0; blue,0}, , line width=1pt](-56.25,35.25) to[short] (-52.5,42.75);
\draw [ color={rgb,255:red,255; green,0; blue,0}, , line width=1pt](-52.5,42.75) to[short] (-53.75,50.25);
\draw [ color={rgb,255:red,255; green,0; blue,0}, , line width=1pt](-53.75,50.25) to[short] (-61.25,54);
\draw [ color={rgb,255:red,255; green,0; blue,0}, , line width=1pt](-61.25,54) to[short] (-70,52.75);
\draw [ color={rgb,255:red,255; green,0; blue,0}, , line width=1pt](-70,52.75) to[short] (-75,47.75);
\node [font=\LARGE, color={rgb,255:red,255; green,0; blue,0}] at (-55.8,39.5) {$\gamma_{\ge l}$};
\end{circuitikz}
}%
\caption{The curve $\gamma$ (in black) and its coarse-grained version $\gamma_{\ge l}$ (in red).}
\label{f6}
\end{figure}

What follows is a heuristic derivation of the result. The key ingredient is a coarse-graining decomposition of curves in $\Gamma_{L,1}$, which is performed with the aim of connecting the expectation $\mathbb{E}I_L$ at two different scales (see~\eqref{m125} below): Given $\gamma\in\Gamma_{L,1}$, let $\gamma_{\ge l}$ be the coarse-grained curve obtained from $\gamma$ by keeping one vertex every $l$. Let $e_1,\ldots,e_N$ be the edges of $\gamma_{\ge l}$. Then $\gamma$ can be reconstructed from $\gamma_{\ge l}$ by replacing each edge $e_n$ with an open curve $\gamma_n$ with the same endpoints (see Figure~\ref{f6}). This gives a decomposition which may be schematically\footnote{This sum could be rigorously defined in terms of currents.} written as
\begin{align}\label{m171}
\gamma=\gamma_{\ge l}+\sum_{n=1}^N\gamma_n.
\end{align}

\medskip

With this notation, the field $W(\cdot)$ is additive. Indeed, one may extend the definition of the winding number to the open curves $\gamma_n$ by setting $w(\gamma_n,z)=w(\tilde\gamma_n,z)$, where $\tilde\gamma_n$ is the closed curve obtained from $\gamma_n$ by connecting the end and the starting point with a segment. By applying the formula in \eqref{m107}, this gives a definition of $W(\gamma_n)$. Then the winding number is additive with respect to the above splitting,
\begin{align*}
&w(\gamma,z)=w(\gamma_{\ge l},z)+\sum_{n=1}^Nw(\gamma_n,z),
\end{align*}
which directly yields additivity of the field $W(\cdot)$,
\begin{align}\label{m167}
W(\gamma)=W(\gamma_{\ge l})+\sum_{n=1}^NW(\gamma_n).
\end{align}

\medskip

Let us come to the core formula \eqref{m169} of the paper, which connects the ratios $W(\gamma)/\len(\gamma)$ and $W(\gamma_{\ge l})/\len(\gamma_{\ge l})$. To this end, we denote by $l_n$ the length of the edge $e_n$ of $\gamma_{\ge l}$ and write the total length of $\gamma$ as
\begin{align}\label{m168}
\len(\gamma)=\len(\gamma_{\ge l})+\sum_{n=1}^N\len(\gamma_n)-l_n.
\end{align}
Combining \eqref{m167} and \eqref{m168}, by purely algebraic steps, we get the iteration
\begin{align}\label{m169}
\frac{W(\gamma)}{\len(\gamma)}-\frac{W(\gamma_{\ge l})}{\len(\gamma_{\ge l})}=\frac{1}{\len(\gamma_{\ge l})}\sum_{n=1}^N W(\gamma_n)-\frac{W(\gamma)}{\len(\gamma)}\big(\len(\gamma_n)-l_n\big).
\end{align}

\medskip

In order to analyze the r.~h.~s.~, we perform a geometric linearization of the perimeter increment $\len(\gamma_n)-l_n$. Hence, we assume that the maximization problem can be solved within the restricted class of 
\begin{align}\label{m71}
\mbox{curves $\gamma$ which are flat up to scale $l$.}
\end{align}
More precisely, we assume that $\gamma_n$ are graphs of functions $h_n$ over $e_n$ with
\begin{align*}
\mbox{$h_n$ piecewise linear over $l$ intervals of length $\gtrapprox1$ and $\text{Lip}(h_n)\ll1$.}
\end{align*}
In this regime, we replace the perimeter increment with the Dirichlet energy\footnote{with the expression "$A_1\lessapprox A_2$ if $B_1\ll B_2$" we mean that for every $\epsilon>0$ there exists $\delta>0$ such that if $B_1\le\delta B_2$ then $A_1\le(1+\epsilon)A_2$. We write $A_1\approx A_2$ if $A_1\lessapprox A_2$ and $A_2\lessapprox A_1$.},
\begin{align}\label{m72}
\len(\gamma_n)-l_n=\int_0^{l_n}dx\big(\sqrt{1+(\frac{dh_n}{dx})^2}-1\big)\approx\int_0^{l_n}\frac{1}{2}(\frac{dh_n}{dx})^2=:D(h_n).
\end{align}
This is a common step in the context of minimal surfaces and goes under the name of harmonic approximation. 

\medskip

After this substitution, we can rewrite the r.~h.~s.~terms in \eqref{m169} as
\begin{align}
&W_n(h_n)-\frac{W(\gamma)}{\len(\gamma)}D(h_n)\quad\mbox{with}\quad W_n(h):=\int_0^{l_n}dx\int_0^{h(x)}dy~\xi_n(x,y),\label{m65}
\end{align}
where $\xi_n$ is obtained from $\xi$ through a rigid change of coordinates (mapping $e_n$ to $[0,l_n]\times\{0\}$), hence still has the law of white noises. Note that the definitions are chosen so that $W_n(h_n)=W(\gamma_n)$. The passage from the length to the Dirichlet energy allows for the homogeneities under vertical rescalings,
\begin{align*}
h=t\hat h\quad\mbox{implies}\quad D(h)=t^2 D(\hat h)\quad\mbox{and}\quad W_n(h)=t^{1/2}\hat W_n(\hat h),
\end{align*}
where $W_n=_{\text{law}}\hat W_n$ in view of the behavior of white noise under the anisotropic rescalings $(x,y)=(\hat x,t\hat y)$. 
Applying them with $t=\big(W(\gamma)/\len(\gamma)\big)^{-2/3}$, naively thinking of it as deterministic, gives
\begin{align}\label{m170}
\frac{W(\gamma)}{\len(\gamma)}-\frac{W(\gamma_{\ge l})}{\len(\gamma_{\ge l})}\approx\frac{\big(W(\gamma)/\len(\gamma)\big)^{-1/3}}{\len(\gamma_{\ge l})}\sum_{n=1}^N \hat W_n(\hat h_n)-D(\hat h_n).
\end{align}

\medskip

The appearance of the actions $\hat W_n-D$ motivates us to consider the following $(1+1)$-dimensional maximization problem: Given $l\ge1$, consider the configuration space
\begin{align*}
\mathcal{H}_{l,1}:=&\{h:[0,l]\to\mathbb{R}~\mbox{piecewise linear on intervals $[x_n,x_{n+1}]$ where}\\ &\mbox{$0=x_0<x_1<\ldots<x_N=l$ with $x_{n+1}-x_{n}\ge1$}\\&\mbox{and such that $h(0)=h(l)=0$}\}.
\end{align*}
We are therefore interested in the one-dimensional action
\begin{align}\label{m163}
A_l:=\max_{h\in\mathcal{H}_{l,1}}(W-D)(h)/l\quad\mbox{and}\quad h^*\in\arg\max_{h\in\mathcal{H}_{l,1}}(W-D)(h)/l
\end{align}
where the Dirichlet energy $D$ and the field term $W$ are defined in \eqref{m72} and \eqref{m65}, respectively. The above action has been shown in \cite{op} to diverge logarithmically in the ratio $l/1$, with a deterministic, leading-order constant. 
In fact, the result in \cite{op} has been proven with a slightly different implementation of the small-scale cutoff defining the configuration space $\mathcal{H}_{l,1}$. The details on how to deduce this precise statement from \cite{op} are given in the appendix. In order to quantify the homogenization result, we need to introduce the following Orlicz norms on the probability space:
\begin{align}\label{m14}
\mbox{For}~s\in[1,\infty), \quad\|X\|_s:=\inf\{N>0:\mathbb{E}\exp(\frac{|X|}{N})^s\le e\}.
\end{align}

\begin{theorem}[{\cite{opw,op}}]\label{T:3}
There exists a deterministic $a^*\in(0,\infty)$ such that\footnote{By $A\lesssim B$ we mean that there is universal constant $C$ such that $A\le CB$. We write $A\sim B$ if $A\lesssim B$ and $B\lesssim A$.}
\begin{align}\label{m75}
\|A_l-a^*\ln l\|_{3/2}\lesssim1\quad\mbox{and}\quad\big\|\max_{x\in[0,l]}|h^*(x)|/l\big\|_3\lesssim 1.
\end{align}
\end{theorem}

\medskip

We can now conclude. Let us pick a maximizer
\begin{align}\label{m56}
\gamma^*\in\arg\max_{\gamma\in\Gamma_{L,1}}\frac{W(\gamma)}{\len(\gamma)}
\end{align}
and suppose that $l$ is small enough so that $\gamma^*$ satisfies \eqref{m71}. Then $\gamma^*_{\ge l}$ has side-length $\gtrapprox l$ hence almost belongs to $\Gamma_{L,l}$. Assuming that the coarse-grained maximizer has almost optimal ratio in the coarse-grained problem, one deduces that 
\begin{align}\label{m173}
\frac{W(\gamma^*_{\ge l})}{\len(\gamma^*_{\ge l})}\approx_{\text{law}} I_{L/l}
\end{align}
Similarly, one may suppose that the rescaled graphs $\hat h^*_n$ are almost maximizer of the local $(1+1)$-dimensional problems, 
\begin{align}
(\hat W_n-\hat D)(\hat h_n^*)/l_n\approx_{\text{law}}A_{l_n}\overset{\eqref{m71}}{\approx}_{\text{law}}A_l.\label{m106}
\end{align}
Substituting into \eqref{m170}, replacing the random variables with their expectations (see the concentration in~\eqref{m130}), and using $\len(\gamma_{\ge l}^*)=\sum_n l_n$, we obtain
\begin{align}\label{m125}
\mathbb{E}I_L-\mathbb{E}I_{L/l}\approx (\mathbb{E}I_L)^{-1/3}\mathbb{E}A_l.
\end{align}
To solve this recurrence, we rearrange the terms and raise to the power $\textstyle{\frac{4}{3}}$,
\begin{align*}
(\mathbb{E}I_{L/l})^{4/3}\approx (\mathbb{E}I_L)^{4/3}\big(1-(\mathbb{E}I_L)^{-4/3}\mathbb{E}A_l\big)^{4/3}\approx (\mathbb{E}I_L)^{4/3}-{\textstyle{\frac{4}{3}}}\mathbb{E}A_l,
\end{align*}
where the last approximation uses $(1+x)^{4/3}\approx1+\textstyle{\frac{4}{3}}x$. Together with Theorem~\ref{T:3}, this suggests that $(\mathbb{E}I_L)^{4/3}$ scales linearly in $\ln L$ and, more precisely, we have discovered the following relation between the asymptotics of the problems in \eqref{m79} and \eqref{m163}.

\begin{theorem}\label{T:1'}
\begin{align*}
\lim_{L\uparrow\infty}\frac{\mathbb{E}I_L}{\ln^{3/4}L}=({\textstyle{\frac{4}{3}}}\lim_{l\uparrow\infty}\frac{\E A_l}{\ln l} )^{3/4}.
\end{align*}
\end{theorem}

\subsection{Size of the fluctuations}

Let us comment on quantifying the convergence rate in Theorem~\ref{T:1}. It aims at describing $\mathbb{E}I_L$ as a function $f(\ln L)$ of the number of scales $\ln L$, where $f$ should be universal, namely independent of the way the small-scale cutoff is implemented. Due to the arbitrariness of the latter, we only know the number of scales up to an $O(1)$ error. Hence, we expect the optimal result to be of the form
\begin{align*}
\mathbb{E}I_L=f(\ln L)+O\big(f'(\ln L)\big),
\end{align*}
where we think of  $O(f'(\ln L))$ as a non-universal error. From Theorem~\ref{T:1} we have that $f(s)\approx i^*s^{3/4}$ for $s\gg1$. Moreover, from the finer \eqref{m125}, 
\begin{align}\label{m129}
&f'(s)\approx f(s)^{-1/3}\approx (i^*)^{-1/3}s^{-1/4}\nonumber\\&\mbox{so that}\quad O(f'(\ln L))=O(\ln^{-1/4}L).
\end{align}
By contrast, we get the quantitative, but not optimal, convergence rate
\begin{align}\label{m131}
\mathbb{E}I_L=i^*\ln^{3/4}L+O(\ln^{3/4-\eta}L)\quad\mbox{for all $\eta\in[0,\textstyle{\frac{1}{12}})$.}
\end{align}
Note that if $f$ grows polynomially, the above heuristic predicts a relative error always of size $O(\ln^{-1}L)$. The asymptotics of the expected value have indeed been computed up to this precision in other critical problems: the maximum of the Gaussian free field in dimension 2 \cite[Theorem 1.2]{bz}; $L^2$ semi-discrete optimal matching \cite[Theorem 1]{gho}, up to a double logarithm; the maximal action of curves in a Brownian environment \cite[Theorem 1]{op}.

\medskip

Let us now discuss the size of the fluctuations of $I_L-\mathbb{E}I_L$. 
For a fixed $\gamma$,
\begin{align}\label{m124}
&\frac{W(\gamma)}{\len(\gamma)}\quad\mbox{is a Gaussian with variance}\nonumber\\&\frac{1}{\len^2(\gamma)}\int dz~w^2(\gamma,z)\le\frac{1}{|Dw|^2(\mathbb{R}^2)}\int dz~w^2(\gamma,z)\lesssim1,
\end{align}
where $|Dw|(\mathbb{R}^2)$ denotes the $BV$ seminorm of $w(\gamma,\cdot)$ and the second inequality follows from the continuous embedding $BV(\mathbb{R}^2)\hookrightarrow L^2(\mathbb{R}^2)$ (see for example~\cite{bp}). Since $I_L$ is a supremum of Gaussians with bounded variance, Borell's inequality gives (cf. \cite[Theorem 4.5.7]{b}) 
\begin{align}\label{m130}
I_L-\mathbb{E}I_L=O(1)\quad\mbox{with Gaussian tails}.
\end{align}
This is, however, not optimal. Indeed, as already noted in \cite{d}, it is possible to infer the size of fluctuations from the asymptotics of $f'$ (see~\eqref{m129}). Adapting those ideas, we establish the following

\begin{theorem}\label{T:5} There exists a universal constant $c>0$ such that
\begin{align}
&\PP (I_L-\E I_L\geq\nu\ln^{-1/4}L)\leq e^{-c\nu}&& \text{for all}\quad\nu\gg \ln\ln L,\label{c88}\\
&\PP (I_L-\E I_L\leq-\nu\ln^{-1/4}L)\leq e^{-e^{c\nu}}&&\text{for all}\quad\ln L\geq \nu\gg \ln\ln L.\label{c89}
\end{align}
\end{theorem}

The double logarithm lower bound on $\nu$ is an artifact of the proof, and is a consequence of the fact that we can only prove \eqref{m125} provided $\ln l\gg\ln\ln L$ (see Lemma~\ref{L:34}). Together with the classical tails in \eqref{m130}, it can be postprocessed to 
\begin{align*}
\text{Var}^{1/2}I_L\lesssim(\ln\ln L)\ln^{-1/4}L.
\end{align*}
Since it is stronger than \eqref{m130}, it can be viewed as a superconcentration result, which is a common phenomenon for the supremum of Gaussian processes, for which we refer to the work of Chatterjee~\cite{c}. We do not know whether the estimates in Theorem~\ref{T:5} are sharp. If so, the right exponential and left double-exponential tails would suggest a limiting Gumbel-type behavior of the fluctuations $I_L-\mathbb{E}I_L$, as established for example for the supremum of the two-dimensional Gaussian free field \cite{bdz}.

\subsection{A variant with relative perimeter}\label{ss:relative}

We note that the variational problem in \eqref{m79} is monotone with respect to inclusion of the underlying domain containing the curves ($B_L$ above). Hence, since every fixed bounded and open set $\Omega\subset\mathbb{R}^2$ can be sandwiched between two balls, one directly deduces that Theorem~\ref{T:1} remains valid when $B_L$ is replaced by the dilation $L\Omega$ of a fixed set $\Omega$. 

\medskip

This blindness with respect to the macroscopic shape of the ambient domain persists also when we consider the following ``relative'' version of the problem, which appears as an approximate dual problem for min-max optimal matching (see Section~\ref{S:context}): Given the open cube $Q_L$ of side $L$ and center 0, we replace $\Gamma_{L,1}$ with the configuration space
\begin{align*}
\Gamma'_{L,1}:=\{&\mbox{closed, piecewise-linear curves $\gamma$ with image $\subset\bar Q_L$}\\&\mbox{and such that $\gamma\cap Q_L$ has sides $\ge1$}\}.
\end{align*}
Let us denote by $\len(\gamma,Q_L)$ the relative length, namely the length of the part of $\gamma$ contained in $Q_L$. Similarly, we introduce the relative version of the field term in \eqref{m107},
\begin{align*}
W(\gamma,Q_L):=\int_{\mathbb{R}^2}w(\gamma,z)\xi(z)-\frac{\int_{\mathbb{R}^2}dz~w(\gamma,z)}{\lambda(Q_L)}\int_{Q_L}\xi,
\end{align*}
where $\lambda$ is the Lebesgue measure, and study the variational problem
\begin{align}\label{m195}
J_L:=\max_{\gamma\in\Gamma'_{L,1}}\frac{W(\gamma,Q_L)}{\len(\gamma,Q_L)}.
\end{align}
If one thinks of $I_L$ as an isoperimetric problem with a random volume term, then $J_L$ would be the corresponding relative isoperimetric problem. The following result states that, unlike its deterministic counterpart, the two problems have the same leading-order asymptotics.

\begin{theorem}\label{T:6}
\begin{align*}
\lim_{L\uparrow\infty}\frac{\mathbb{E}J_L}{\ln^{3/4}L}=\lim_{L\uparrow\infty}\frac{\mathbb{E}I_L}{\ln^{3/4}L}.
\end{align*}
\end{theorem}

The proof of this theorem is essentially the same as that of Theorem~\ref{T:1}, hence we only highlight the differences in Subsections~\ref{SS:rel} and \ref{SS:bound-reg}. Let us remark that the convenient choice of the domain $Q_L$ should not be fundamental and could be replaced by any dilation $L\Omega$ of a fixed smooth $\Omega$, up to some small modification of the proof.

\section{Motivations and context}\label{S:context}

\subsection{Min-max optimal matching}

Our interest in the variational problem \eqref{m79} arises from the following connection with optimal matchings of random point clouds. Let $\lambda$ denote the Lebesgue measure in $\mathbb{R}^2$ restricted to $Q_L$ and let $X_1,\ldots,X_N$ be independent and uniformly distributed in $Q_L$ with $N:=L^2$ an integer. The goal is to find a partition $\{F_i\}_i$ of $Q_L$, with tiles $F_i$ of unit volume,  which minimizes
 $\max_i\max_{z\in F_i}|z-X_i|$. Equivalently, this can be thought of as an optimal transportation problem and defines the $W_\infty$-distance between the purely atomic $\mu:=\sum_{i=1}^N\delta_{X_i} $ and $\lambda$,
 \begin{align*}
 W_\infty(\mu,\lambda):=\inf\{ \sup_{x,y\in\text{spt}\,\pi}|x-y|\mid \pi\text{ coupling of $\mu$ and $\lambda$} \}.
 \end{align*}
The major contribution to understanding its scaling in $L$ has been made by Leighton and Shor \cite{ls}; they replaced $\lambda$ with the atomic measure $\sum_i\delta_{Y_i}$ with $Y_i$ in a uniform grid, thus looking at 
\begin{align*}
\min\{\max_{i=1}^N|X_i-Y_{\pi(i)}|~\big|~\mbox{permutation $\pi$ of $\{1,\ldots,N\}$}\}.
\end{align*}
For this reason, it is called the min-max matching problem. They have discovered the following scaling of $W_\infty(\mu,\lambda)$ in dimension 2,

\begin{theorem}[Leighton--Shor]\label{T:LS}
For $L\gg1$,
\begin{align*}
\E W_\infty(\mu,\lambda)\sim\ln^{3/4}L.
\end{align*}
\end{theorem}

Let us explain how the Gaussian discrepancy problem \eqref{m195} arises from the $W_\infty$-distance between $\mu$ and $\lambda$. For a more detailed explanation and a rigorous proof, we refer to \cite{copw}. Given any set $\Sigma\subset Q_L$, it is necessary that, in order to compensate for the excess of points $(\mu-\lambda)(\Sigma)$ of $\{X\}$ relative to $\lambda(\Sigma)$, there is enough volume at distance $R_L:=W_\infty(\mu,\lambda)$ from $\Sigma$:
\begin{align}\label{c:108}
(\mu-\lambda)(\Sigma)
\leq\lambda(\Sigma^{R_L})-\lambda(\Sigma)
\end{align}
where we let $\Sigma^R:=\{ x\in\mathbb{R}^2\mid \mbox{dist}(x,\Sigma)\leq R \}$.
If $\Sigma$ is regular up to scale $\gg R_L$, we may approximate the r.~h.~s.~by\footnote{For a smooth set $\Sigma$, we denote by $P(\Sigma,Q_L)$ the length of the part of the boundary contained in $Q_L$.} $R_LP(\Sigma,Q_L)$; while if $\Sigma$ is flat on scales $\gg 1$,  we may approximate the centered shot noise in the l.~h.~s.~by a white noise with a volume correction: 
\begin{align*}
(\mu-\lambda)(\Sigma)\approx\int_\Sigma\xi-\frac{\lambda(\Sigma)}{\lambda(Q_L)}\int_{Q_L}\xi=:W(\Sigma,Q_L).
\end{align*}
This suggests that 
\begin{align}\label{m196}
&R_L\approx \sup_{\Sigma\subset \mathcal{P}_{L,1}}\frac{W(\Sigma, Q_L)}{P(\Sigma,Q_L)}
\end{align}
where the microscopic scale $1$ given by the typical point distance is retained by the lower bound of the side lengths of $\Sigma$
\begin{align*}
\mathcal{P}_{L,1}:=\{\mbox{$\Sigma\subset Q_L$ is a polygon and the edge lengths of $\partial\Sigma\cap Q_L$ are $\ge1$}\}.
\end{align*}
Note that the problem appearing on the r.~h.~s.~of \eqref{m196} would coincide with the definition of $J_L$, if one restricts to the class of simple curves $\gamma$. In Lemma~\ref{L:23}, it is shown that this restriction does not alter the asymptotic behavior of the problem and indeed we have the following

\begin{corollary}\label{Cor:3}
\begin{align*}
\lim_{L\uparrow\infty}\frac{1}{\ln^{3/4}L}\mathbb{E}\sup_{\Sigma\in\mathcal{P}_{L,1}}\frac{W(\Sigma,Q_L)}{P(\Sigma,Q_L)}=({\textstyle{\frac{4}{3}}}\lim_{l\uparrow\infty}\frac{\E A_l}{\ln l} )^{3/4}.
\end{align*}
\end{corollary}
In this sense, the optimization problem $J_L$ considered in \eqref{m195} is a Gaussian isoperimetric approximation of the  min-max matching problem. This approximation is made rigorous in the work \cite{copw} and together with the previous corollary, it refines the asymptotics given by Leighton-Shor's theorem.

\subsection{Large-scale regularity of random curves on the plane}
The passage from the area enlargement to the perimeter in \eqref{c:108}, as well as the geometric linearization  \eqref{m72} in this paper, requires regularity of the optimal set $\Sigma^*$. This is fully justified in this work, following the strategy in Ried and Wagner \cite{rw} by recognizing the optimal set as a perimeter almost-minimizer. We explain the connection with \cite{rw} starting from the  problem 
\begin{align}\label{c:109}
\Sigma^*\in\arg\max\{ \frac{W(\Sigma)}{P(\Sigma)}\mid \Sigma \subset B_L\text{ polygon with side-length}\geq 1 \},
\end{align}
where $P(\Sigma)$ is the perimeter of $\Sigma$ and $W(\Sigma):=\int_\Sigma\xi$.
Note that this corresponds to the problem $I_L$ when one restricts to simple curves $\gamma=\partial\Sigma$; see Section \ref{SS:Simple} for the connection between the two.
The fundamental observation is that $\Sigma^*$ minimizes the  additive functional 
\begin{align}\label{m197}
\Sigma\mapsto P(\Sigma)-\epsilon W(\Sigma)\quad\text{where}\quad \epsilon:=(\frac{W(\Sigma^*)}{P(\Sigma^*)})^{-1}.
\end{align}
By Theorem~\ref{T:LS}, the relevant regime is $\epsilon\sim\ln^{-3/4}L\ll1$; It suggests that the field term $W(\Sigma)$ can be seen as a small perturbation of the perimeter and consequently that $\partial\Sigma^*$ is regular up to some mesoscopic scale $l\gg1$. The above additive functional has been studied in \cite{rw} for a fixed $\epsilon\ll1$, albeit with a slightly different small-scale cutoff.  We refer to Section \ref{SS:regularity} for much more detail about how this strategy is adjusted and implemented in our setting. We point out that we choose to work with curves $\gamma$ instead of sets, because it simplifies the regularity theory, without affecting the leading-order asymptotics of the problem (cf.~Corollary~\ref{Cor:3}).

\subsection{Correlation length in random-field Ising model}

A discrete counterpart of \eqref{m197} is the  random-field Ising model in two dimensions. One replaces $\mathbb{R}^2$ by $\mathbb{Z}^2 $ and the characteristic function of $\Sigma\subset\mathbb{R}^2$ by spin configurations  $\sigma:\mathbb{Z}^2\to\{0,1\}$. Then mimicking the perimeter by attractive nearest-neighbor spin interaction and white noise on $\mathbb{R}^2$ by  $\{\xi(z)\}_{z\in\mathbb{Z}^2}$ independent standard Gaussians, we are led to consider the Hamiltonian,
\begin{align*}
P(\sigma)-\epsilon W(\sigma):=\sum_{|z-z'|=1}|\sigma(z)-\sigma(z')|
-\epsilon\sum_z\xi(z)\sigma(z).
\end{align*}
Aizenman and Wehr \cite{aw} proved that in two dimensions at zero temperature, almost surely, there exists a unique local\footnote{namely, with respect to finite-volume perturbations} minimizer of the Hamiltonian above  in the whole space. Ding and Wirth \cite{dw} gave a more quantitative description of this phenomenon: Considering the minimizer $\sigma^*$ of the Hamiltonian on a finite sublattice $Q_L\cap \mathbb{Z}^2$ with zero boundary conditions, one expects that the influence of the boundary fades away, to the effect that $\lim_{L\to\infty}\E \sigma^*(0)=1/2$. They showed that $\E\sigma^*(0)$ becomes larger than $1/3$ when the system size $L$ is linked to the tuning parameter $\epsilon$ by 
\begin{align*}
\ln^{3/4} L\sim\epsilon^{-1}.
\end{align*}
Their argument is closely related to the maximal ratio $I_L$ defined in \eqref{m79}, with the curves $\gamma$ constrained to lie on the grid of spacing 1, hence named ``lattice animals''. Indeed, $\E\sigma^*(0)\geq 1/3$ implies that with a sizable probability, there is a non-trivial $\{\sigma^*=1\}$ phase; on the other hand, such a configuration  can improve on the constant configuration only if the random-field gain of $ \{\sigma^*=1\}$ is comparable to its boundary cost. Hence, we see that   $\epsilon^{-1}$ should be comparable to the ratio $\max_{\sigma}W(\sigma)/P(\sigma)$. In \cite[Proposition 2.2]{dwarxiv}, they also introduced the isotropic version of this problem as in \eqref{c:109}, with the appropriate extension to non-simple curves using the winding number; As explained in Section~\ref{S:introduction}, we find this to be a more natural formulation and thus we will take this point of view.

\subsection{Maximum of two-dimensional Gaussian free field}
We may compare our main results with other critical extremal-value  statistics of Gaussian fields. The by far better-understood is the maximum of the two-dimensional discrete Gaussian free field. Let $\{\eta_z\}_{z\in Q_L\cap\mathbb{Z}^2}$ be the GFF in a square of side length $L$ with zero boundary conditions, and let
\begin{align*}
M_L:=\max_{z\in Q_L\cap \mathbb{Z}^2}\eta_z.
\end{align*}
Bramson and Zeitouni \cite{bz} proved that
\begin{align}\label{m198}
\mathbb{E}M_L
=\textstyle{2\sqrt{2/\pi}\ln L-\frac{3}{4}\sqrt{2/\pi}\ln\ln L+O(1),}
\end{align}
which was then improved to a convergence in law of the centered maximum to a randomly shifted Gumbel distribution (see~\cite{bdz}). Indeed, Ding \cite{d} was the first to  show that the fluctuations have exponential right tail and double exponential left tail. For the problem $I_L$, which is a supremum of a Gaussian process indexed by the infinite-dimensional set of planar polygonal curves, we obtain analogous results. Although we are far from having an expansion of $\E I_L$ with the same precision as \eqref{m198}, we expect that Theorem~\ref{T:5} gives almost-optimal tails estimates for the fluctuations $I_L-\E I_L$.

\medskip

The similarity of the two problems is not surprising since they both rely on a multiscale mechanism and the leading-order asymptotic comes from contribution of logarithmically many scales. We remark that our proof of Theorem \ref{T:5}, drawing on the techniques developed in \cite{d}, takes advantage of the optimal expansion of the energy increment between two close scales $\E I_L-\E I_{L/l}$. In particular, the double-exponential estimates come from the fact that an unusually small maximum requires simultaneous suppression of exponentially many independent almost-maximizers. 

\section{Strategy of the proof}

\subsection{Large-scale regularity of the optimal curve}\label{SS:regularity}

In order to run the argument given in the introduction, in particular ensuring the assumption \eqref{m71}, we need to establish a regularity theory for the maximizer $\gamma^*$. This is done by showing that, with high probability, the oscillations of the normal $\nu^*$ of $\gamma^*$ are small up to scales $l$ satisfying $\ln l\ll\ln^{1/8}L$. The precise statement, which is the goal of this subsection, is the following

\begin{lemma}\label{L:18}
If $\ln l\gg \ln\ln L $, then outside an event of probability $\leq L^{-1}$,
\begin{align*}
\max_{|z-z'|\leq l}|\nu^*(z)-\nu^*(z')|\lesssim \ln l/\ln^{1/8}L.
\end{align*}
\end{lemma}

Let us explain the main ideas of the proof. We first observe that
\begin{align}\label{m57}
\mbox{$\gamma^*$ minimizes $\len(\gamma)-\epsilon W(\gamma)$}\quad\mbox{with}\quad\epsilon:=I_L^{-1}.
\end{align}
It is deduced through the following algebraic steps. For any $\gamma\in\Gamma_{L,1}$,
\begin{align*}
\frac{ W(\gamma^*)}{\len(\gamma^*)}\overset{\eqref{m56}}{\ge}\frac{ W(\gamma)}{\len(\gamma)}=\frac{ W(\gamma^*)+\big( W(\gamma)- W(\gamma^*)\big)}{\len(\gamma^*)+\big(\len(\gamma)-\len(\gamma^*)\big)}.
\end{align*}
Multiplying by the denominator of the r.~h.~s.~and subtracting $ W(\gamma^*)$,
\begin{align*}
I_L\big(\len(\gamma)-\len(\gamma^*)\big)\overset{\eqref{m56}}{=}\frac{ W(\gamma^*)}{\len(\gamma^*)}\big(\len(\gamma)-\len(\gamma^*)\big)\ge W(\gamma)- W(\gamma^*).
\end{align*}
Rearranging the terms, one obtains \eqref{m57}.

\medskip

The regularity of minimizer of the additive problem \eqref{m57} has already been studied in \cite{rw} for a deterministic $\epsilon\ll1$ and we will closely follow their strategy by recognizing $\gamma^*$ as a quasi-minimizer of $\len(\cdot)$.  To ensure the quasi-minimality of $\gamma^*$, we need some a priori bounds on $I_L$, which are provided in the following lemma. Even though it would follow by adapting the ideas of Leighton-Shor \cite{ls} and Talagrand \cite{t22}, we will give a self-contained argument and show that it is one of the outcomes of our argument. In the appendix, we display an efficient proof of the lower bound which only uses the $(1+1)$-dimensional problem in \eqref{m163}. The concentration estimate has already been proven in~\eqref{m130}. We postpone the upper bound until the end of the proof and establish it via a continuity argument, together with the convergence in Theorem~\ref{T:1}.

\begin{lemma}\label{T:7}
\begin{align*}
\mathbb{E}I_L\sim\ln^{3/4}L\quad\mbox{and}\quad\|I_L-\mathbb{E} I_L\|_2\lesssim1.
\end{align*}
\end{lemma}

Let us now recall the main ideas in \cite{rw}. The first step is to exploit the fact that $\epsilon\ll1$ to treat the field term in \eqref{m57} as a perturbation and hence deduce that $\gamma^*$ is a local quasi-minimizer of the length $\len(\cdot)$. Let $B_l(z)$ be the ball of radius $l$ and center $z\in\mathbb{R}^2$. Pick a curve $\gamma$ which differs from $\gamma^*$ in $B_l(z)$, and more precisely such that one may decompose the latter as
\begin{align*}
\gamma^*=\gamma+\tilde\gamma\quad\mbox{with $\tilde\gamma\in\Gamma_{L,1}$ having image in $B_l(z).$}
\end{align*}
This provides the splitting of the field term and the control on the length: 
\begin{align*}
W(\gamma^*)=W(\gamma)+W(\tilde\gamma)\quad\mbox{and}\quad|\len(\gamma^*)-\len(\gamma)|\le\len(\tilde\gamma).
\end{align*}
Let us denote by $\len(\gamma, B_l(z))$ the length of the part of $\gamma$ contained in $B_l(z)$. From the minimality in \eqref{m57} we have
\begin{align*}
\len(\gamma^*,B_l(z))-\len(\gamma,B_l(z))=\len(\gamma^*)-\len(\gamma)\le\epsilon\big( W(\gamma^*)- W(\gamma)\big).
\end{align*}
Extending the notation in \eqref{m79} to
\begin{align*}
I_l(z):=\mbox{the maximal ratio in the ball $B_l(z)$ with small-scale 1,}
\end{align*}
the difference of the noises satisfies
\begin{align*}
 W(\gamma^*)- W(\gamma)=W(\tilde\gamma)&\le I_l(z)\len(\tilde\gamma)\\&\le I_l(z)\big(\len(\gamma^*,B_l(z))+\len(\gamma,B_l(z))\big).
\end{align*}
Putting the last two inequalities together one obtains the quasi-minimality
\begin{align}
&\len(\gamma^*,B_l(z))\le(1+\eta_l)\len(\gamma,B_l(z))\nonumber\\&\mbox{with}\quad\eta_l\lesssim{I}_l(z)/{I}_L\quad\mbox{provided}\quad{I}_l(z)/{I}_L\ll1.\label{m165}
\end{align}

\medskip

A rigorous version of \eqref{m165} is given in Lemma~\ref{L:12}, where the competitor $\gamma$ is given by replacing a branch of $\gamma^*$ with a simple straight line. To describe it, it is convenient to express a curve in terms of its arc-length parametrization. Given $\gamma\in \Gamma_{L,1}$ with vertices $\{z_n\}_{n=1}^N$, we let $\mathbb{S}$ be the circle with length $\len(\gamma)$, so that $\gamma:\mathbb{S}\to \mathbb{R}^2$ gives the arc-length parametrization of $\gamma$. We also define the vertex times of $\gamma$,
\begin{align*}
\mathbb{S}\supset V(\gamma):=\{ t_1,\ldots, t_N\}\quad\mbox{such that}\quad\gamma(t_n)=z_n~\mbox{for}~n=1,\ldots,N.
\end{align*}

\medskip

Given $\gamma\in\Gamma_{L,1}$, a time $t\in\mathbb{S}$ and a length $l>0$, define the branch
\begin{align}\label{c:90}
\gamma_{t,l}:=\mbox{the restriction of $\gamma$ to $[t^{in}_{t,l},t^{out}_{t,l}]$}\quad\mbox{(see Figure~\ref{f1})}
\end{align}
where $[t^{in}_{t,l},t^{out}_{t,l}]$ is the largest connected subset of $\mathbb{S}$ such that
\begin{align*}
&\mbox{$\gamma_{t,l}\subset B_l(\gamma(t))$,}\quad\mbox{$t_{t,l}^{in}, t, t_{t,l}^{out}$ lie in this order on $\mathbb{S}$}\quad\mbox{and}\quad
t_{t,l}^{in},t_{t,l}^{out}\in V(\gamma).
\end{align*}
We will call $t_{t,l}^{in}$ and $t_{t,l}^{out}$ entry and exit times, respectively. Note that they are well-defined as long as $l\ge3$ and $\gamma$ is not completely contained inside $B_l(\gamma(t))$. In the following Lemma~\ref{L:12}, this is ensured by giving a lower bound on the diameter of $\gamma$. 

\medskip

Given two points $ z_1,z_2\in\mathbb{R}^2 $, we will denote by $[z_1,z_2]$ the line segment between $z_1$ and $z_2$. Replacing the branch $\gamma_{t,l}$ by the segment $[\gamma^*(t_{t,l}^{in}),\gamma^*(t_{t,l}^{out})]$ (or an  appropriate modification of it in the case $|\gamma^*(t_{t,l}^{in})-\gamma^*(t_{t,l}^{out})|<1$) provides us with a competitor in the class $\Gamma_{L,1}$.

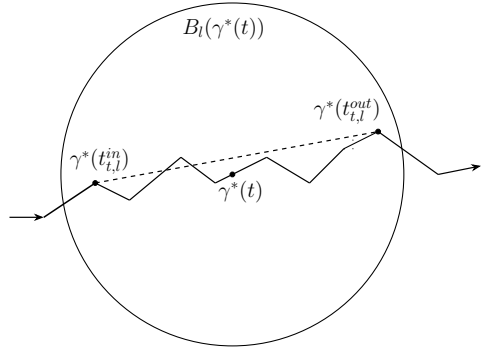
\begin{figure}
\centering
\resizebox{0.5\textwidth}{!}{%
\begin{circuitikz}[scale=0.8]
\tikzstyle{every node}=[font=\LARGE]
\draw [ line width=0.7pt ] (4.25,12.25) circle (5cm);
\draw [line width=1pt, short] (-1.25,11) -- (0.25,12);
\draw [line width=0.8pt, short] (0.25,12) -- (1.25,11.5);
\draw [line width=0.8pt, short] (1.25,11.5) -- (2.75,12.75);
\draw [line width=0.8pt, short] (2.75,12.75) -- (3.75,12);
\draw [line width=0.8pt, short] (3.75,12) -- (5.25,12.75);
\draw [line width=0.8pt, short] (5.25,12.75) -- (6.5,12);
\draw [line width=0.8pt, short] (6.5,12) -- (7.5,13);
\draw [line width=0.7pt, short] (7.5,13) -- (8.5,13.5);
\draw [line width=0.8pt, short] (8.5,13.5) -- (10.25,12.25);
\draw [short] (7.5,13) -- (7.5,13);
\draw [short] (7.5,13) -- (7.5,13);
\draw [short] (7.75,13) -- (7.75,13);
\draw [short] (7.75,13.25) -- (7.75,13.25);
\draw [line width=0.8pt, ->, >=Stealth] (-2.25,11) -- (-1.25,11);
\draw [line width=0.8pt, ->, >=Stealth] (10.25,12.25) -- (11.5,12.5);
\node [font=\Large] at (7.6,14.1) {$\gamma^*(t^{out}_{t,l})$};
\node at (0.25,12) [circ] {};
\node at (8.5,13.5) [circ] {};
\node [font=\Large] at (0.38,12.65) {$\gamma^*(t^{in}_{t,l})$};
\node at (4.25,12.25) [circ] {};
\node [font=\Large] at (4.5,11.75) {$\gamma^*(t)$};
\draw [line width=0.8pt, dashed] (0.25,12) -- (8.5,13.5);

\node [font=\Large] at (4,16.5) {$B_l(\gamma^*(t))$};
\end{circuitikz}
}%
\caption{Construction of the branch $\gamma^*_{t,l}$ and of the (dotted) segment $[\gamma^*(t^{in}_{t,l}),\gamma^*(t^{out}_{t,l})]$.}
\label{f1}
\end{figure}


\begin{lemma}\label{L:12}
Let $l\ge3$ be such that 
\begin{align}\label{c:24}
\sup_{z\in B_L}I_l(z)/I_L\ll 1.
\end{align}
Then the diameter of $\gamma^*$ is $\ge2l$ and for any time $t\in\mathbb{S}$, we have
\begin{align*}
\len(\gamma_{t,l}^*)\leq (1+\eta_l)|\gamma^*(t^{in}_{t,l})-\gamma^*(t^{out}_{t,l})|\quad\text{with}\quad \eta_l\lesssim\sup_{z\in B_L} I_l(z)/ I_L.
\end{align*}
\end{lemma}

In view of the previous lemma, we now give a stochastic estimate of the ratio in \eqref{c:24}. The denominator $ I_L$ is controlled from below by the lower bound  and the concentration in Lemma~\ref{T:7}. For the numerator, we use the stationarity of $z\mapsto I_l(z)$ to write
\begin{align*}
\sup_{z\in B_L} I_l(z)=\mathbb{E} I_l+\sup_{z\in B_L}\big( I_l(z)-\mathbb{E} I_l(z)\big).
\end{align*}
Then both terms are estimated by Lemma~\ref{T:7}.

\begin{lemma}\label{L:7}
Let $l\ge 3$. Then outside an event with probability $\leq  L^{-1}$, we have
\begin{align*}
\sup_{z\in B_L} I_l(z)/ I_L\lesssim\frac{\ln^{3/4}l+\ln^{1/2}(L/l)}{\ln^{3/4}L}
\end{align*}
\end{lemma}

\begin{remark}\label{rk:3}The above estimate is reduced to 
\begin{align}\label{c:2}
\sup_{z\in B_L} I_l(z)/ I_L\lesssim\ln^{-1/4}L.
\end{align}
in the regime $\ln l\ll \ln^{2/3}L$. This will always be the case, since in order to exploit the regularity result in Lemma~\ref{L:18}, we will assume $\ln l\ll\ln^{1/8}L$.  

\end{remark}

In the next lemma, we derive a regularity result for curves $\gamma$ that are almost minimizers of the perimeter (up to a multiplicative factor $\eta$ and for all scales $l'$ less than a fixed mesoscopic scale $l$, cf. \eqref{m54}). As in \cite{rw}, we exploit some ideas from the theory of minimal surfaces in this simpler two-dimensional setting.

\begin{lemma}\label{L:8}
Let $0<\eta\ll 1$, $l\ge3$. Suppose that a curve $\gamma\in\Gamma_{L,1}$ has diameter $\ge 2l$ and 
\begin{align}\label{m54}
\len(\gamma_{t,l'})\le(1+\eta)|\gamma(t^{in}_{t,l'})-\gamma(t^{out}_{t,l'})|\quad \text{for any $t\in\mathbb{S}$ and $3\le l'\le l$}.
\end{align}
Then for any\footnote{We recall that we are assuming $\gamma:\mathbb{S}\to\mathbb{R}^2$ being parametrized by arc-length.} $s,s'\in \mathbb{S}$
\begin{align*}
|\nu(s)-\nu(s')|\lesssim \sqrt{\eta}\ln (e+|s-s'|).
\end{align*}
\end{lemma}

\begin{remark}\label{rk:1}
Note that the following partial converse is also true. Let $T=\len(\gamma)$ and $ \gamma:[0,T]\to\mathbb{R}^2 $ be a (possibly open) Lipschitz curve parametrized by arc length  satisfying 
\begin{align*}
|\nu(s)-\nu(s')|\leq\sqrt{\eta}\ll1\quad\text{for all}\quad s,s'\in[0,T].
\end{align*}
Then
\begin{align}\label{c:45}
&\len(\gamma)-|\gamma(0)-\gamma(T)|\lesssim \sqrt{\eta}|\gamma(0)-\gamma(T)|,\\
\label{c:47}
&|\nu(s)-\nu^{line}|\leq 2\sqrt{\eta},\quad\text{for all}\quad s\in[0,T]
\end{align}
where $ \nu^{line}:=$unit normal of the segment $ [\gamma(0),\gamma(T)] $.
\end{remark}

Notice that the modulus of continuity of the normal $\nu$ in Lemma~\ref{L:8} is given with respect to the arc length $|s-s'|$. Using \eqref{c:45}, we have that it is comparable to the (a priori smaller) chord length $|
\gamma(s)-\gamma(s')|$. This is true as long as one picks $s,s'$ on the same branch $\gamma_{t,l}$, with the length $l$ satisfying
\begin{align*}
\sqrt{\eta}\ln l\ll1\quad\mbox{so that by Lemma~\ref{L:8}}\quad|\nu(s)-\nu(s')|\ll1.
\end{align*}

\begin{corollary}\label{Cor:2}
Let $0<\eta\ll1,l\gg1$ be such that
\begin{align*}
\sqrt{\eta}\ln l\ll1.
\end{align*}
Suppose that $\gamma\in\Gamma_{L,1}$ has diameter $\ge 2l$ and satisfies \eqref{m54}, then for any $t\in\mathbb{S}$
\begin{align*}
|\nu(s)-\nu(s')|\lesssim \sqrt{\eta}\ln (e+|\gamma(s)-\gamma(s')|)\quad\mbox{for all}~ s,s'\in[t^{in}_{t,l},t^{out}_{t,l}].
\end{align*}
\end{corollary}

As a direct consequence of Lemmas~\ref{L:12} and~\ref{L:7} (in the simplified form of Remark~\ref{rk:3}) and Corollary~\ref{Cor:2}, we obtain a regularity result for $\gamma^*$ on single branches $\gamma_{t,l}$.

\begin{corollary}\label{Cor:1}
Let   $l\gg 1$  be such that 
\begin{align}\label{m154}
\ln l\ll \ln ^{1/8}L.
\end{align}
Then outside an event with probability $ \leq L^{-1} $, for any $ t\in\mathbb{S}$ and $s,s'\in [t^{in}_{t,l},t^{out}_{t,l}]$ 
\begin{align*}
|\nu^*(s)-\nu^*(s')|\lesssim\frac{\ln(e+|\gamma^*(s)-\gamma^*(s')|)}{\ln^{1/8}L}.
\end{align*}
\end{corollary}

\medskip

Let $l$ be a length scale satisfying \eqref{m154}. 
To conclude the proof of Lemma \ref{L:18},  we still need to show closeness of the normals $\nu^*(s)$ and $\nu^*(s')$ when $|\gamma^*(s)-\gamma^*(s')|\ll l$,  even if they lie on different branches. To this end, we construct a competitor by rewiring $\gamma^*$ inside the ball $B_l(z)$ with $z=\frac{1}{2}(\gamma^*(s)+\gamma^*(s'))$, as shown in Figure~\ref{f2}. As in Lemma~\ref{L:12}, $\gamma^*$ has almost minimal length, which implies that the two branches have to be almost parallel. More details can be found in Section~\ref{s:details}. 

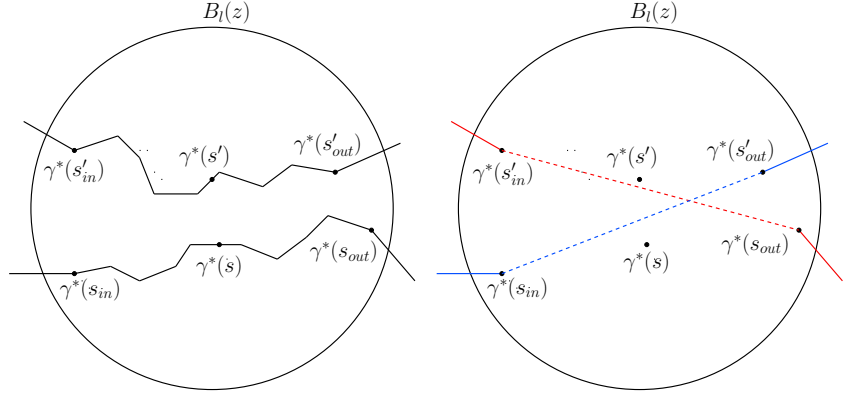
\begin{figure}[!ht]
\centering
\resizebox{0.9\textwidth}{!}{%
\begin{circuitikz}[scale=0.8]
\tikzstyle{every node}=[font=\large]
\draw [ line width=0.9pt ] (12.5,10.5) circle (6.25cm);
\node [font=\LARGE] at (-2.5,12.3) {$\gamma^*(s')$};
\node [font=\LARGE] at (-1.75,17.25) {$B_{l}(z)$};
\draw [line width=0.9pt, short] (-4.5,12.5) -- (-4.5,12.5);
\draw [line width=0.9pt, short] (-4.75,12.5) -- (-4.75,12.5);
\draw [line width=0.9pt, short] (-8.75,13.5) -- (-7,12.5);
\draw [line width=0.9pt, short] (-7,12.5) -- (-5.5,13);
\draw [line width=0.9pt, short] (-5.5,13) -- (-4.75,12.25);
\draw [line width=0.9pt, short] (-4,11.5) -- (-4,11.5);
\draw [line width=0.9pt, short] (-4.25,11.75) -- (-4.25,11.75);
\draw [line width=0.9pt, short] (-4.75,12.25) -- (-4.25,11);
\draw [line width=0.9pt, short] (-4.25,11) -- (-2.75,11);
\draw [line width=0.9pt, short] (-2.75,11) -- (-2,11.75);
\draw [line width=0.9pt, short] (-2,11.75) -- (-0.5,11.25);
\draw [line width=0.9pt, short] (-0.5,11.25) -- (0.5,12);
\draw [line width=0.9pt, short] (0.5,12) -- (2,11.75);
\draw [line width=0.9pt, short] (2,11.75) -- (4.25,12.75);
\draw [line width=0.9pt, short] (5.5,18.75) -- (5.5,18.75);
\node at (-2.25,11.5) [circ] {};
\draw [ line width=0.9pt](-9.25,8.25) to[short] (-7,8.25);
\draw [line width=0.9pt, short] (-6.5,7.75) -- (-6.5,7.75);
\draw [line width=0.9pt, short] (-6.75,8) -- (-6.75,8);
\draw [line width=0.9pt, short] (-7,8.25) -- (-5.75,8.5);
\draw [line width=0.9pt, short] (-5.75,8.5) -- (-4.75,8);
\draw [line width=0.9pt, short] (-4.75,8) -- (-3.5,8.5);
\draw [line width=0.9pt, short] (-3,9.25) -- (-3.5,8.5);
\draw [line width=0.9pt, short] (-3,9.25) -- (-1.25,9.25);
\draw [line width=0.9pt, short] (-1.25,9.25) -- (0,8.75);
\draw [line width=0.9pt, short] (1,9.5) -- (1.75,10.25);
\draw [line width=0.9pt, short] (1.75,10.25) -- (3.25,9.75);
\draw [line width=0.9pt, short] (3.25,9.75) -- (4.75,8);
\draw [line width=0.9pt, short] (-1.75,8.75) -- (-1.75,8.75);
\node at (-2,9.25) [circ] {};
\draw [line width=0.9pt, short] (0,8.75) -- (1,9.5);
\node [font=\LARGE] at (-2,8.5) {$\gamma^*(s)$};
\draw [ line width=0.9pt ] (-2.25,10.5) circle (6.25cm);
\node [font=\LARGE] at (-7,11.8) {$\gamma^*(s'_{in})$};
\node [font=\LARGE] at (1.75,12.7) {$\gamma^*(s'_{out})$};
\node [font=\LARGE] at (2.3,9) {$\gamma^*(s_{out})$};
\node [font=\LARGE] at (-6.45,7.7) {$\gamma^*(s_{in})$};
\node at (-7,12.5) [circ] {};
\node at (2,11.75) [circ] {};
\node at (-7,8.25) [circ] {};
\node [font=\LARGE] at (12.25,12.25) {$\gamma^*(s')$};
\node [font=\LARGE] at (13,17.25) {$B_{l}(z)$};
\draw [line width=0.9pt, short] (10.25,12.5) -- (10.25,12.5);
\draw [line width=0.9pt, short] (10,12.5) -- (10,12.5);
\draw [ color={rgb,255:red,250; green,0; blue,0}, line width=0.9pt, short] (6,13.5) -- (7.75,12.5);
\draw [line width=0.9pt, short] (10.75,11.5) -- (10.75,11.5);
\draw [line width=0.9pt, short] (10.5,11.75) -- (10.5,11.75);
\draw [ color={rgb,255:red,0; green,76; blue,255}, line width=0.9pt, short] (16.75,11.75) -- (19,12.75);
\draw [line width=0.9pt, short] (20.25,18.75) -- (20.25,18.75);
\node at (12.5,11.5) [circ] {};
\draw [ color={rgb,255:red,0; green,76; blue,255}, , line width=0.9pt](5.5,8.25) to[short] (7.75,8.25);
\draw [line width=0.9pt, short] (8.25,7.75) -- (8.25,7.75);
\draw [line width=0.9pt, short] (8,8) -- (8,8);
\draw [ color={rgb,255:red,245; green,0; blue,0}, line width=0.9pt, short] (18,9.75) -- (19.5,8);
\draw [line width=0.9pt, short] (13,8.75) -- (13,8.75);
\node at (12.75,9.25) [circ] {};
\node [font=\LARGE] at (12.75,8.65) {$\gamma^*(s)$};
\node [font=\LARGE] at (7.8,11.7) {$\gamma^*(s'_{in})$};
\node [font=\LARGE] at (16,12.5) {$\gamma^*(s'_{out})$};
\node [font=\LARGE] at (16.5,9.25) {$\gamma^*(s_{out})$};
\node [font=\LARGE] at (8.3,7.75) {$\gamma^*(s_{in})$};
\node at (7.75,12.5) [circ] {};
\node at (16.75,11.75) [circ] {};
\node at (7.75,8.25) [circ] {};
\node at (3.25,9.75) [circ] {};
\node at (18,9.75) [circ] {};
\draw [ color={rgb,255:red,250; green,0; blue,0}, line width=0.9pt, dashed] (7.75,12.5) -- (18,9.75);
\draw [ color={rgb,255:red,0; green,76; blue,255}, line width=0.9pt, dashed] (7.75,8.25) -- (16.75,11.75);
\end{circuitikz}
}%
\caption{On the left, the optimal curve $\gamma^*$. On the right, the two competitors $\gamma_1$ (in red) and $\gamma_2$ (in blue) used to prove Lemma~\ref{L:18}. The dotted lines are $\tilde\gamma_1$ and $\tilde\gamma_2$.}
\label{f2}
\end{figure}

\subsection{Discretization of the local reference frames}

From the regularity theory of the previous subsection, we are now able to justify the geometric linearization in \eqref{m72}. Starting from this subsection, we deal with the following statistical errors, which arise from the local $(1+1)$-dimensional problems: In the definition \eqref{m65}, we naively thought of $\xi_n$ as new instances of the white noise $\xi$, since they are given by rigid transformations, which depend on $e_n$. Nonetheless, since the latter is random when we evaluate \eqref{m170} at $\gamma=\gamma^*$, the previous argument is not rigorous.

\medskip

The first step to treat these statistical errors is to discretize the possible changes of coordinates, namely the possible realizations of the sides $e_n$. To this end, we pick a non-dimensional parameter $\kappa$, which will be fixed in the sequel, and introduce the projection\footnote{here $[x]$ denotes the integer part of $x$.} onto the grid $\kappa l\mathbb{Z}^2$
\begin{align}\label{e7}
\mathbb{R}^2\ni z=(x,y)\mapsto\bar z=(\bar x,\bar y)\in\kappa l\mathbb{Z}^2\quad\mbox{with}\quad\bar x:=\kappa l[\frac{x}{\kappa l}],\bar y:=\kappa l[\frac{y}{\kappa l}].
\end{align}
By acting on the two endpoints, this extends to a projection of segments $e\mapsto\bar e$.  In view of this discretization, we need the following technical lemma, which shows that the results of Theorem~\ref{T:3} still hold when we allow for an error in the choice of the Dirichlet boundary conditions. More precisely, we pick a segment $\bar e$ and show that all the infinitely many $(1+1)$-dimensional problems, which are built upon segments $e$ with $e\mapsto\bar e$, satisfy uniform error estimates. 

\medskip

For convenience, we state separately upper and lower bounds and we stick to the geometrically nonlinear setting, but assuming, for the upper bound, that the normal of the curves has small oscillations (measured by $\theta$). The harmonic approximation is performed inside the proof, as reflected by the relative error $\theta^2$ in the upper bound. Since we cannot rely on the homogeneity of the Dirichlet energy, we need an additional parameter $I$, measuring the relative strength of the perimeter and the field term. Even though in the introduction we chose the random $I=W(\gamma)/\len(\gamma)$ (cf.~\eqref{m65}), we will assume for the time being that the latter is deterministic and address this problem in the next subsection. A sketch of the proof is given in Figure~\ref{f5}.

\begin{lemma}\label{L:30}
Let us fix a segment $\bar e=[\bar z_1,\bar z_2]$ of length $\bar l:=|\bar z_1-\bar z_2|\gg1$ and non-dimensional parameters $0\le\theta,I^{-1}\ll1$,  $0\le\kappa\le\min\{\ln^{-1}\bar l,I^{-4/3}\}$. Then there exist two non-negative random variables $\delta^\pm$ with $\|\delta^\pm\|_{3/2}\lesssim1$ such that for every $z_1,z_2$ with $|z_i-\bar z_i|\le\kappa\bar l$ the following hold:
\begin{itemize}
\item[$(i)$] For every curve $\gamma$ connecting $z_1$ and $z_2$ with side-length $\ge 1$ and normal satisfying $\sup_{s,s'}|\nu(s)-\nu(s')|\le\theta$, we have
\begin{align*}
&W(\gamma)-I\big(\len(\gamma)-|z_1-z_2|\big)\\&\le\big(1+\delta^+(\ln^{-1}\bar l+\theta^2)\big)a^*I^{-1/3}|z_1-z_2|\ln\bar l.
\end{align*}
\item[($ii)$] There exists a $\gamma$ connecting $z_1$ and $z_2$ with side-length $\ge 1$ such that
\begin{align*}
&W(\gamma)-I\big(\len(\gamma)-|z_1-z_2|\big)\\&\ge\big(1-\delta^-\ln^{-1}\bar l\big)a^*I^{-1/3}|z_1-z_2|\ln\bar l.
\end{align*}
\end{itemize}
\end{lemma}

\begin{figure}
\centering
\resizebox{1\textwidth}{!}{%
\begin{circuitikz}
\tikzstyle{every node}=[font=\LARGE]
\draw [ line width=1pt , dashed] (-25,24) circle (2.5cm);
\draw [ line width=1pt , dashed] (-8.75,24) circle (2.5cm);
\draw [ line width=1pt](-25,24) to[short] (-8.75,24);
\node [font=\LARGE] at (-25.25,23.5) {$\bar z_1$};
\node [font=\LARGE] at (-8.75,23.5) {$\bar z_2$};
\node [font=\LARGE] at (-17,23) {$\gamma$};
\node at (-25,24) [circ] {};
\node at (-8.75,24) [circ] {};
\node at (-30,24) [circ] {};
\node at (-3.75,24) [circ] {};
\node [font=\LARGE] at (-30.25,23.5) {$\tilde z_1$};
\node [font=\LARGE] at (-3.75,23.5) {$\tilde z_2$};

\node [font=\LARGE] at (-26,25.75) {$z_1$};
\node [font=\LARGE] at (-7.5,25.25) {$z_2$};
\node at (-26.25,25.5) [circ] {};
\node at (-7.5,24.75) [circ] {};
\node [font=\LARGE] at (-25,21) {$B_{\kappa \bar l}(\bar z_1)$};
\node [font=\LARGE] at (-8.5,21) {$B_{\kappa \bar l}(\bar z_2)$};
\draw [ line width=1pt](-25.5,24.75) to[short] (-24.25,24.75);

\draw [ line width=1pt](-24.25,24.75) to[short] (-23.25,24.5);
\draw [ line width=1pt](-23.25,24.5) to[short] (-22.25,25);
\draw [ line width=1pt](-22.25,25) to[short] (-21,24.75);
\draw [ line width=1pt](-21,24.75) to[short] (-20.25,24.25);
\draw [ line width=1pt](-20.25,24.25) to[short] (-19.5,24.5);
\draw [ line width=1pt](-19.5,24.5) to[short] (-18.25,24.25);
\draw [ line width=1pt](-18.25,24.25) to[short] (-17,23.5);
\draw [ line width=1pt](-17,23.5) to[short] (-16,23.25);
\draw [ line width=1pt](-16,23.25) to[short] (-15,23.5);
\draw [ line width=1pt](-15,23.5) to[short] (-13.75,23);
\draw [ line width=1pt](-13.75,23) to[short] (-12.25,23.5);
\draw [ line width=1pt](-12.25,23.5) to[short] (-11.5,22.75);
\draw [ line width=1pt](-11.5,22.75) to[short] (-10.5,23.5);
\draw [ line width=1pt](-10.5,23.5) to[short] (-10,24.25);
\draw [ line width=1pt](-10,24.25) to[short] (-8.75,25);
\draw [ line width=1pt](-26.25,25.5) to[short] (-25.5,24.75);
\draw [ line width=1pt](-7.5,24.75) to[short] (-8.75,25);
\draw [ color={rgb,255:red,255; green,0; blue,0}, , line width=1pt](-30,24) to[short] (-26.25,25.5);
\draw [ color={rgb,255:red,255; green,0; blue,0}, , line width=1pt](-7.5,24.75) to[short] (-3.75,24);
\end{circuitikz}
}%
\caption{Illustration of proof the upper bound. Given a curve $\gamma$ with varying boundary conditions $z_1,z_2$, by adding the two red segments one constructs a competitor for the problem with fixed boundary conditions $\tilde z_1,\tilde z_2$.}
\label{f5}
\end{figure}
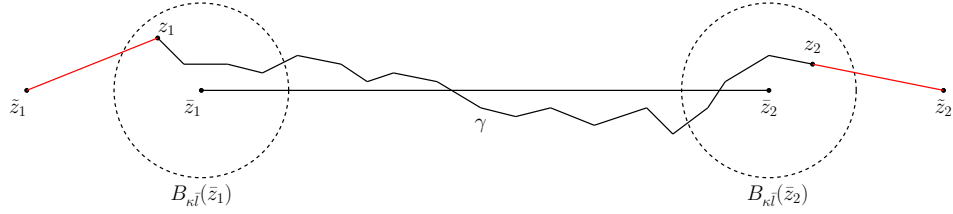

\subsection{Upper and lower bounds}\label{SS:upperlower-bounds}

The goal of this subsection is to obtain the recursive formula \eqref{m125} up to some stochastic errors, which we control in the next subsections. To this end, we follow the strategy outlined in the introduction by combining the algebraic formula~\eqref{m169}, the regularity theory as in Lemma~\ref{L:18} and the estimates of the $(1+1)$-dimensional problem from Lemma~\ref{L:30}. In view of the uncontrolled assumption that coarse-graining and maximization commute (cf.~\eqref{m173}), we proceed slightly differently for the upper and the lower bound. 

\medskip

In view of Lemma~\ref{L:18} and the concentration in Lemma~\ref{T:7}, we know that there exists a good event $A$ with
\begin{align}\label{m183}
\mathbb{P}(A^c)\lesssim L^{-1},
\end{align}
such that the following regularity and concentration estimates hold:
\begin{align}
&|\nu^*(z)-\nu^*(z')|\lesssim\theta:=\ln l/\ln^{1/8}L\quad\mbox{for all}~|z-z'|\le l;\label{m166}\\
&|I_L-\E I_L|\lesssim\ln^{1/2}L. \label{c:101}
\end{align}
In order to apply Lemma \ref{L:30}, we shall choose $l$ and $\kappa$ such that
\begin{align}\label{c:100}
\theta\ll1\quad\mbox{and}\quad \kappa\ll\ln^{-1}L.
\end{align}

\medskip

Let us come to the upper bound, in the form of Lemma~\ref{L:31} below. Evaluating \eqref{m169} at $\gamma=\gamma^*$ gives
\begin{align}\label{c:105}
I_L-\frac{W(\gamma^*_{\ge l})}{\len(\gamma^*_{\ge l})} \leq \frac{1}{\len(\gamma^*_{\geq l})}\sum_{n=1}^N W(\gamma_n^*)-I_L(\len(\gamma_n^*)-l_n).
\end{align}
In view of the regularity in \eqref{m166}, $\gamma^*_{\ge l}$ has lower-bounded side-length, 
\begin{align*}
\gamma^*_{\geq l}\in \Gamma_{L,l/2}\quad\mbox{so that}\quad \frac{W(\gamma^*_{\ge l})}{\len({\gamma^*_{\ge l}})}\le\bar I_{2L/l}=_{\text{law}}I_{2L/l}.
\end{align*}
This allows us to simplify the l.~h.~s.~. Let us come to the r.~h.~s.~. We denote the edges of $\gamma^*_{\ge l}$ by $e_n^*$ and their projections by $\bar e_n^*$. We now rely on the concentration in \eqref{c:101} and on the monotonicity of the r.~h.~s.~with respect to $I_L$, to replace it with the deterministic $I:=\mathbb{E}I_L-C\ln^{1/2}L$ for some universal $C$. Together with the regularity \eqref{m166}, this enables us to apply Lemma \ref{L:30} $(i)$ and obtain
\begin{align*}
I_L-\bar I_{2L/l} \leq \frac{1}{\len(\gamma^*_{\geq l})}\sum_{n=1}^N &\big(1+\delta^+(\bar e_n^*) (\ln^{-1} l+\theta^2) \big)\times \\&\times a^*\big(\E I_L+O(\ln^{1/2}L)\big)^{-1/3}l_n \ln \len(\bar e_n^*),
\end{align*}
where the collection of random variables $\{\delta^+(\bar e)\}$ satisfies $\|\delta^+(\bar e)\|_{3/2}\lesssim1$ for every deterministic $\bar e$. Now, up to multiplying $\delta^+(\bar e^*_n)$ by an $O(1)$ factor, we can replace on the r.~h.~s.~$\len(\bar e_n^*)$ by $l$ and neglect $O(\ln^{1/2}L)$, since
\begin{align}
&\len(\bar e^*_n)\sim l\quad\mbox{so that}\quad\ln \len(\bar e_n^*)=\big(1+O(\ln^{-1}l)\big)\ln l;\notag\\
&\big(\E I_L+O(\ln^{1/2}L)\big)^{-1/3}\overset{\eqref{c:100},\text{Lemma~\ref{T:7}}}{=}(\E I_L)^{-1/3}\big(1+O(\ln^{-1}l)\big).\label{e6}
\end{align}
Recalling $\len(\gamma_{\geq l}^*)=\sum_{n=1}^N l_n$, we have deduced
\begin{align*}
I_L- \bar I_{2L/l} \leq \Big(1+(\ln^{-1}l+\theta^2) \big(\sum_{n=1}^N l_n\big)^{-1}\sum_{n=1}^N l_n\delta^+(\bar e_n^*) \Big)a^*(\E I_L)^{-1/3}\ln l.
\end{align*}
Using that $l_n\sim l$, one can replace the weighted mean of $\delta^+$ by an arithmetic mean. We may also replace $l$ by $2l$ and absorb the $O(1)$ error from $\ln (2l)=\ln l+O(1)$ into $\delta^+$.  Thus  we recover the following result.

\begin{lemma}\label{L:31}
If \eqref{m166}, \eqref{c:101} and \eqref{c:100} hold, then
\begin{align*}
&I_L- \bar I_{L/l} \leq \Big(1+(\ln^{-1}l+\theta^2) N^{-1}\sum_{n=1}^N\delta^+(\bar e_n^*) \Big)a^*(\E I_L)^{-1/3}\ln l\\
&\mbox{where}\quad \bar I_{L/l}=_{\text{law}}I_{L/l}\mbox{ and $\|\delta^+(\bar e)\|_{3/2}\lesssim 1$ for every fixed $\bar e$.}
\end{align*}
\end{lemma}

Similarly, let us treat the lower bound, which is deduced from the construction of an appropriate competitor for the variational problem $I_L$. Let $\gamma^*_l$ be the optimizer within the class $\Gamma_{L,l}$ and we denote now by $e_n^*$ its edges. Let us introduce
\begin{align*}
\gamma:=\gamma^*_l+\sum_{n=1}^N\gamma_n\in\Gamma_{L,1}
\end{align*}
where for every $n$, $\gamma_n$ is the curve connecting the endpoints of $e_n^*$ given by Lemma \ref{L:30} with the choice $I:=\E I_L+C\ln^{1/2}L$. Now applying \eqref{m169} to $\gamma$, 
\begin{align}
&I_L-\bar I_{L/l}\geq \frac{1}{\len(\gamma_l^*)}\sum_{n=1}^NW(\gamma_n)-\frac{W(\gamma)}{\len(\gamma)}(\len(\gamma_n)-l_n)\nonumber\\
&\mbox{where}\quad\bar I_{L/l}:=\frac{W(\gamma^*_l)}{\len(\gamma^*_l)}=_{\text{law}}I_{L/l}.\label{c:106}
\end{align}
Again by monotonicity, one may replace $W(\gamma)/\len(\gamma)$ on the r.~h.~s.~by $I_L\le\mathbb{E}I_L+O(\ln^{1/2}L)$ so that  Lemma \ref{L:30} $(ii)$ with $I:=\mathbb{E}I_L+C\ln^{1/2}L$ yields
\begin{align*}
I_L-\bar I_{L/l}\geq \frac{1}{\len(\gamma_l^*)}\sum_{n=1}^N&\big(1-\delta^-(\bar e_n^*)\ln^{-1}l \big)\times\\&\times a^*\big(\E I_L+O(\ln^{1/2}L)\big)^{-1/3}l_n \ln \len(\bar e_n^*).
\end{align*}
By the same reasoning as in \eqref{e6}, we may drop $O(\ln^{1/2}L)$ and replace $\ln \len(\bar e_n^*) $ by $\ln l$  up to multiplying $\delta^-(e_n^*)$ by an $O(1)$ factor, obtaining

\begin{lemma}\label{L:32}
If \eqref{c:101} holds, then
\begin{align}\label{c:104}
&I_L-\bar I_{L/l}\geq \big(1-(\ln l)^{-1}N^{-1}\sum_{n=1}^N\delta^-(\bar e_n^*) \big)a^*(\E I_L)^{-1/3} \ln l\nonumber\\
&\mbox{where}\quad\bar I_{L/l}=_{\mbox{law}}I_{L/l}\mbox{ and $\|\delta^-(\bar e)\|_{3/2}\lesssim1 $ for every fixed $\bar e$.}
\end{align}
\end{lemma}

Starting from Lemmas~\ref{L:31} and \ref{L:32}, we would like to take expectations on both sides. To this end, we apply Lemma~\ref{T:7} to estimate the contribution of the bad event $A^c$,
\begin{align}\label{m194}
\mathbb{E}I_L I(A^c)\lesssim(\mathbb{E}^{1/2}I_L^2)\mathbb{P}^{1/2}(A^c)\lesssim (\ln^{3/4}L)L^{-1},
\end{align}
which can be absorbed into the other errors. 

\begin{lemma}\label{L:33}
Assume that
\begin{align}\label{m184}
\theta=\ln l/\ln^{1/8}L\ll1\quad\mbox{and}\quad\ln l\gg\ln\ln L.
\end{align}
Then
\begin{align*}
\mathbb{E}I_L-\mathbb{E}I_{L/l}=\Big(1+O\big((\ln^{-1}l+\theta^2)\mathbb{E}N^{-1}\sum_{n=1}^N\delta^\pm(\bar e_n^*)\big)\Big)a^*(\mathbb{E}I_L)^{-1/3}\ln l.
\end{align*}
\end{lemma}

\subsection{Control of the statistical error}\label{SS:statistics}

Note that, if we want to successfully convert Lemma~\ref{L:33} into
\begin{align*}
\mathbb{E}I_L-\mathbb{E}I_{L/l}\approx a^*(\mathbb{E}I_L)^{-1/3}\ln l,
\end{align*}
we need to show that the relative error is negligible,
\begin{align}\label{m180}
\big(\ln^{-1}l+\theta^2\big)\mathbb{E} N^{-1}\sum_{n=1}^N\delta^\pm(\bar e_n^*)\ll1.
\end{align}
A simple strategy might be to apply a sup bound on the l.~h.~s.
\begin{align*}
\mathbb{E}N^{-1}\sum_{n=1}^N\delta^\pm(\bar e_n^*)\le\mathbb{E}\sup_{\bar e}\delta^\pm(\bar e)\lesssim\ln^{2/3}L,
\end{align*}
where in the last step we relied on \eqref{m83} and we used that there are $\lesssim (L/\kappa l)^4$ edges $\bar e$ with endpoints on the grid $\kappa l\mathbb{Z}^2$. This estimate is unfortunately too crude, since in view of the regularity \eqref{m166} we are in the regime $\ln l\ll\ln^{1/8}L$.

\medskip

Let us now give a modified strategy, which avoids the supremum bound and exploits the fact that averages are better-behaved, in terms of concentration properties, than suprema. We introduce the net in the configuration space 
\begin{align*}
\bar\Gamma:=\{\mbox{curves $\bar\gamma$ with image in $B_L$ and vertices in $\kappa l\mathbb{Z}^2$}\},
\end{align*}
which comes with a projection map $\gamma\mapsto\bar\gamma$ obtained by projecting each vertex. This net is the natural index set for the above errors: Given a curve $\bar\gamma$ with edges $\bar e_n$ we set $\delta_n(\bar\gamma):=\delta^\pm(\bar e_n)$ and
\begin{align*}
\delta(\bar\gamma):=N^{-1}\sum_{n=1}^N\delta_n(\bar\gamma).
\end{align*}

\medskip

The aim is now to estimate $\delta(\bar\gamma)$ uniformly over $\bar\gamma\in\bar \Gamma$. To this end, we work under the additional hypothesis,
\begin{align}\label{m174}
\mbox{for every fixed curve $\bar\gamma$, the variables $\{\delta_n(\bar\gamma)\}_n$ are independent},
\end{align}
the justification of which will require the next subsection. Let us assume, up to scaling, that $\|\delta_n(\bar\gamma)\|_{3/2}\le1$. Combining \eqref{m174} with H\"older and Chebyshev inequalities, we get the large deviation bound for a single $\bar\gamma$,
\begin{align}\label{m179}
\mathbb{P}\big(N^{-1}\sum_{n=1}^N\delta_n(\bar\gamma)\ge\nu\big)\le\mathbb{P}\big(N^{-1}\sum_{n=1}^N\delta_n(\bar\gamma)^{3/2}\ge \nu^{3/2}\big)\le e^{-N\nu^{3/2}+N}.
\end{align}

\medskip

Before taking the supremum over $\bar\gamma$, let us make two observations, which help in reducing the number of elements in $\bar\Gamma$ that we have to consider.
The first observation is that we can assume a lower bound on the number of vertices on $\bar\gamma$. Indeed, because of the previous subsection, we are only interested in the case where
\begin{align}\label{m178}
\mbox{$\bar\gamma$ is the projection of either $\gamma^*_{\ge l}$ or $\gamma^*_l$}.
\end{align}
From Lemma~\ref{L:12} and the estimate in Lemma~\ref{L:7}, we have that
$\gamma^*_{\ge l}$ and $\gamma^*_l$ have diameter $\gtrsim l'$ provided $\ln l'\ll \ln L$, namely
\begin{align*}
\mbox{the diameters of $\gamma^*_{\ge l}$ and $\gamma^*_l$ are $\gtrsim L^\alpha$ for some $\alpha\in(0,1)$.}
\end{align*}
Since $\ln l\ll\ln^{1/8}L$, up to slightly decreasing $\alpha$, we can assume
\begin{align}\label{m176}
\mbox{$\bar\gamma$ has $N \gtrsim L^\alpha$ vertices for some universal $\alpha\in(0,1)$.}
\end{align}
The second observation is that \eqref{m178} implies the upper bound
\begin{align}\label{m177}
\mbox{$\bar\gamma$ has side-length $\lesssim l$.}
\end{align}

\medskip

Combining \eqref{m176} and \eqref{m177}, we have the counting estimate
\begin{align*}
\ln \#\{\bar\gamma\in\bar\Gamma~\mbox{with $N$ vertices and satisfying \eqref{m176} $\&$ \eqref{m177}}\}\lesssim N\ln\kappa^{-1},
\end{align*}
which follows from noting that we have $\lesssim(L/\kappa l)^2$ many choices for the first vertex of $\bar\gamma$ and $\lesssim(l/\kappa l)^2$ many for each of the remaining $N-1$ vertices. Together with \eqref{m179} and the union bound, it yields
\begin{align*}
\ln\mathbb{P}\Big(\sup_{\bar\gamma\in\bar\Gamma}\big\{&N^{-1}\sum_{n=1}^N\delta_n(\bar\gamma)~|~\mbox{$\bar\gamma$ with $N$ vertices, satisfying \eqref{m176} $\&$ \eqref{m177}}\big\}\ge\nu\Big)\\&\le-N\nu^{3/2}+N+CN\ln\kappa^{-1}\lesssim-N\nu^{3/2}\quad\mbox{provided $\nu\gg\ln^{2/3}\kappa^{-1}$}.
\end{align*}
Since the probabilities are summable in $N$, we have obtained the following.

\begin{lemma}\label{L:35}
Assume that \eqref{m174} holds. Then
\begin{align*}
\mathbb{E}\sup\{N^{-1}\sum_{n=1}^N\delta_n(\bar\gamma)~|~\mbox{$\bar\gamma\in\bar\Gamma$ satisfying \eqref{m176} $\&$ \eqref{m177}}\}\lesssim\ln^{2/3}\kappa^{-1}.
\end{align*}
\end{lemma}

We observe that this estimate suffices: Substituting it in \eqref{m180} yields the lower bound for $\kappa$,
\begin{align}\label{m193}
(\ln^{-1}l+\theta^2)\ln^{2/3}\kappa^{-1}\ll1.
\end{align}
which is compatible with the upper bound \eqref{c:100}. In fact, any choice of the form $\kappa=\ln^{-C}L$ for some large constant $C$ would work, provided $\ln l\gg\ln \ln L$ (condition already required for the regularity in Lemma~\ref{L:18} to hold).

\subsection{An almost-maximizing simple curve}\label{SS:Simple} The goal of this subsection is to justify the assumption \eqref{m174}. It will be a consequence of the following two additional conditions, which we will then ensure: For every segment $\bar e$, the two variables
\begin{align}\label{m181}
\mbox{$\delta^\pm(\bar e)$ depend only on $\xi$ restricted to $B_{4l}(z)$ with $z:=$midpoint of $\bar e$;}
\end{align}
For every possible $\bar\gamma$ and $z$ in its image,
\begin{align}\label{m182}
\mbox{$\bar\gamma\cap B_{8l}(z)$ consists of one curve.}
\end{align}
Using the spatial independence of the white noise $\xi$, the previous two conditions and the regularity of $\bar\gamma$ at scale $l$, one deduces a relaxed version of \eqref{m174}, where the random variables $\{\delta_n(\bar\gamma)\}_n$ are independent after having been divided into a finite, universally controlled, number of groups. This is anyway sufficient to conclude.

\medskip

Here comes the justification of \eqref{m181}. For the variables $\delta^+(\bar e)$ this is immediate from Lemma~\ref{L:30} $(i)$, since $\kappa,\theta\ll1$ and the fact that $\len(\bar e)\le(4+O(\kappa))l$. For the random variables $\delta^-(\bar e)$, we need an additional argument. By inspecting the proof of Lemma~\ref{L:30} $(ii)$ and applying the $L^\infty$-estimate from Theorem~\ref{T:3}, one deduces the additional property of the competitors $\gamma$,
\begin{align*}
\big\|\sup_{z_1,z_2}\mathrm{dist}(\gamma,\bar e)\big\|_3\lesssim I^{-2/3}\len(\bar e)\lesssim I^{-2/3}l.
\end{align*}
Using that $I=\mathbb{E}I_L+O(\ln^{1/2}L)\gtrsim\ln^{3/4}L$, the estimate \eqref{m83} and the fact that there are $\lesssim (L/(\kappa l))^4$ many edges $\bar e$, we deduce the uniform estimate
\begin{align*}
&\sup_{\bar e}\sup_{z_1,z_2}\mathrm{dist}(\gamma,\bar e)\lesssim I^{-2/3}l(\ln^{1/3}L+\delta)\quad\mbox{with}\quad\|\delta\|_3\le1,\\
&\mbox{and therefore}\quad\sup_{\bar e}\sup_{z_1,z_2}\mathrm{dist}(\gamma,\bar e)\lesssim l\ln^{-1/6}L
\end{align*}
up to slightly changing the event $A$ from Subsection~\ref{SS:upperlower-bounds} by assuming that $\delta\lesssim\ln^{1/3}L$, but retaining the estimate \eqref{m183}. Hence \eqref{m181} holds also for $\delta^-$.

\medskip

Let us conclude with the property in \eqref{m182}, which we consider for $\gamma^*$ and transfer to $\bar\gamma$ by using regularity and $\kappa\ll1$. Note that it would imply that $\gamma^*$ is simple (it cannot self-intersect) and in fact it is equivalent for regular curves, as shown in the following lemma.

\begin{lemma}\label{L:27}
If $\gamma$ is a simple, closed curve satisfying for some $l$,
\begin{align*}
|\nu(s)-\nu(s')|\ll1\quad\mbox{for every $s,s'$ with}\quad|\gamma(s)-\gamma(s')|\le l,
\end{align*}
then $\gamma\cap B_l(z)$ consists of only one curve for every $z$ in the image of $\gamma$.
\end{lemma}

We observe that, if it were not for the small-scale cutoff, we could restrict the maximization problem for $I_L$ to simple curves from the very beginning. Indeed, one may decompose any curve $\gamma$ as
\begin{align*}
&\gamma=\gamma_{\text{loop}}+\gamma'
\end{align*}
where $\gamma_{\text{loop}}$ is a simple, closed, subloop of $\gamma$ and
\begin{align*}
\len(\gamma)=\len(\gamma_{\text{loop}})+\len(\gamma')\quad\mbox{and}\quad W(\gamma)= W(\gamma_{\text{loop}})+W(\gamma').
\end{align*}
Since $\gamma'$ is a competitor for $I_L$ we have that
\begin{align*}
\frac{W(\gamma')}{\len(\gamma')}\le\frac{W(\gamma)}{\len(\gamma)}=\frac{W(\gamma_{\text{loop}})+W(\gamma')}{\len(\gamma_{\text{loop}})+\len(\gamma')}\quad\mbox{so that}\quad\frac{W(\gamma)}{\len(\gamma)}\le\frac{W(\gamma_{\text{loop}})}{\len(\gamma_{\text{loop}})},
\end{align*}
where in the last implication we used $\frac{a}{b}\le\frac{a+c}{b+d}\Longrightarrow\frac{a+c}{b+d}\le\frac{c}{d}$. The last inequality implies that the supremum over simple curves coincides with the overall supremum.

\medskip

In our case, the previous argument unfortunately does not work, since simple subloops of $\gamma^*$ may not satisfy the small-scale cutoff, hence may not be competitors. This difficulty can be overcome by slightly modifying one of the subloops of $\gamma^*$ to a simple, closed curve $\gamma^*_{\text{loop}}$, which is regular and an almost maximizer of $I_L$. In Figure~\ref{f3} there is a picture of such a modification.

\begin{lemma}\label{L:9} Assume that $l,L\gg1$ satisfy  \eqref{m184}. Then, outside of an event of probability $\lesssim L^{-1}$, there exists a simple curve $\gamma^*_{\text{loop}}\in\Gamma_{L,1}$ with normal $\nu$ such that
\begin{align*}
\max_{|z-z'|\leq l}|\nu(z)-\nu(z')|\lesssim\ln l/\ln^{1/8}L\quad\mbox{and}\quad I_L-\frac{W(\gamma^*_{\text{loop}})}{\len(\gamma^*_{\text{loop}})}\lesssim\frac{1}{l}.
\end{align*}
\end{lemma}

\begin{figure}[!ht]
\centering
\resizebox{0.95\textwidth}{!}{%
\begin{circuitikz}
\tikzstyle{every node}=[font=\LARGE]
\draw [ line width=0.9pt , dashed] (-3,10.5) circle (5.25cm);
\draw [ dashed] (-15.75,8) circle (2.75cm);
\draw [->, >=Stealth] (-22.25,10.25) -- (-20.5,9.5);
\draw [short] (-12.75,9.25) -- (-11.75,10);
\draw [short] (-16.5,8.25) -- (-17.75,8);
\draw [short] (-16.5,8.25) -- (-15,7.25);
\draw [short] (-13.5,8.25) -- (-12.75,9.25);
\draw [short] (-11.5,11.25) -- (-12,12.5);
\draw [short] (-16.25,14.75) -- (-17.5,15.25);
\draw [short] (-12,12.5) -- (-12,13.5);
\draw [short] (-11.5,11.25) -- (-11.75,10);
\draw [short] (-12.75,14.5) -- (-12,13.5);
\draw [short] (-16.25,14.75) -- (-15.25,14.5);
\draw [->, >=Stealth] (-13.75,15) -- (-15.25,14.5);
\draw [short] (-12.75,14.5) -- (-13.75,15);
\draw [short] (-20.5,9.5) -- (-20,8.75);
\draw [short] (-20,8.75) -- (-19,8.25);
\draw [short] (-19,8.25) -- (-17.75,8);
\draw [->, >=Stealth] (-15,7.25) -- (-13.5,8.25);
\draw [short] (-17.5,15.25) -- (-19,15);
\draw [short] (-19,15) -- (-20.5,14.25);
\draw [short] (-20.5,14.25) -- (-21,13);
\draw [short] (-21,13) -- (-21,11.75);
\draw [->, >=Stealth] (-21,11.75) -- (-20.25,10.5);
\draw [short] (-20.25,10.5) -- (-19.5,9.25);
\draw [short] (-19.5,9.25) -- (-18.25,8.75);
\draw [short] (-18.25,8.75) -- (-17,8.75);
\draw [short] (-17,8.75) -- (-16,8.25);
\draw [short] (-16,8.25) -- (-15.25,6.75);
\draw [short] (-15.25,6.75) -- (-14.5,6.5);
\draw [->, >=Stealth] (-14.5,6.5) -- (-12.75,6.25);
\draw [short] (-12.75,6.25) -- (-11.75,5);
\draw [short] (-14,6.5) -- (-14,6.5);
\draw [short] (-20.75,11.25) -- (-20.75,11.25);
\draw [short] (-20.5,11) -- (-20.5,11);
\node [font=\LARGE] at (-22,9.75) {$\gamma^*$};
\draw [dashed] (-16,10.75) -- (-4.75,15.5);
\draw [dashed] (-16,5.25) -- (-2.75,5.25);
\node at (-15.75,7.75) [circ] {};
\draw [line width=0.9pt, short] (-4.75,10.75) -- (-7.5,10.5);
\draw [line width=0.9pt, short] (-4.75,10.75) -- (-1.5,10.25);
\draw [line width=0.9pt, short] (-10,11.5) -- (-7.5,10.5);
\draw [line width=0.9pt, short] (-1.5,10.25) -- (0.25,11.75);
\draw [line width=0.9pt, short] (-9.25,13.25) -- (-6.5,13.25);
\draw [line width=0.9pt, short] (-6.5,13.25) -- (-4.25,12);
\draw [line width=0.9pt, short] (-4.25,12) -- (-1.75,9);
\draw [line width=0.9pt, short] (-1.75,9) -- (0,7.75);
\draw [line width=0.9pt, ->, >=Stealth] (0,7.75) -- (3,7.25);
\draw [line width=0.9pt, short] (-1.75,9) -- (-1.75,9);
\draw [line width=0.9pt, ->, >=Stealth] (0.25,11.75) -- (2.75,12.5);
\node at (-3,10.5) [circ] {};
\node [font=\LARGE] at (-3.5,10) {$\gamma^*(t_0)$};
\node [font=\LARGE] at (-6,13.75) {$\gamma^*(t_1^-)$};
\node [font=\LARGE] at (0.25,12.25) {$\gamma^*(t_0^+)$};
\node [font=\LARGE] at (-7.25,11) {$\gamma^*(t_0^-)$};
\node [font=\LARGE] at (0,7.25) {$\gamma^*(t_1^+)$};
\node at (-7.5,10.5) [circ] {};
\node at (-6.5,13.25) [circ] {};
\node at (0.25,11.75) [circ] {};
\node at (0,7.75) [circ] {};
\draw [ color={rgb,255:red,255; green,0; blue,0}, line width=0.9pt, dashed] (-6.5,13.25) -- (0.25,11.75);
\draw [ color={rgb,255:red,17; green,24; blue,228}, line width=0.9pt, dashed] (-7.5,10.5) -- (0,7.75);
\draw [dashed] (-5.75,7.5) -- (-5.75,7.5);
\end{circuitikz}
}%
\caption{Construction of $\gamma^*_{\text{loop}}$ (in red) and $\gamma'$ (in blue).}
\label{f3}
\end{figure}
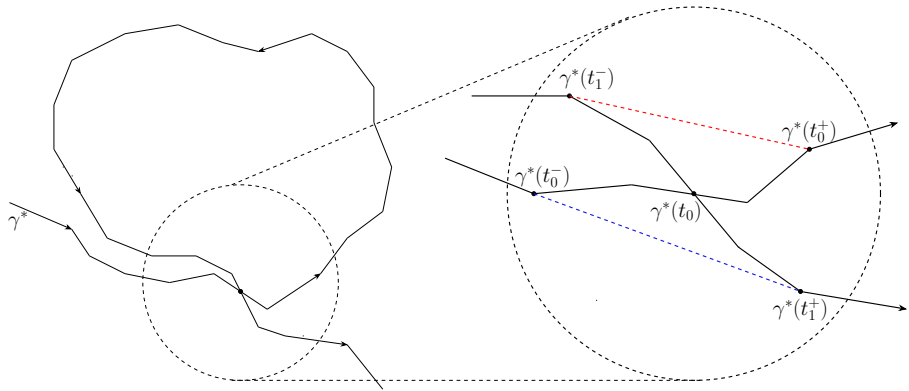

\medskip

Now we show how to combine the results from this and the previous subsections to obtain 
\begin{lemma}\label{L:34}
Assume that \eqref{m184} holds. Then
\begin{align*}
\mathbb{E}I_L-\mathbb{E}I_{L/l}\approx a^*(\mathbb{E}I_L)^{-1/3}\ln l.
\end{align*}
\end{lemma}

Indeed, if we can show Lemma \ref{L:33} with random errors $\{\delta^\pm(\bar e_n)\}$ satisfying the estimate in Lemma \ref{L:35}, then the conclusion follows immediately. In view of Lemmas \ref{L:27} and \ref{L:9}, the assumption of Lemma \ref{L:35} is satisfied if we could rerun the argument for deriving Lemma \ref{L:31} and \ref{L:32} with $\gamma^*$ replaced with $\gamma^*_{\text{loop}}$ and $\gamma^*_{l}$ replaced with $\gamma^*_{\text{loop},l}$, the almost maximizing simple loop in the class $\Gamma_{L,l}$. In doing so, there are two additional errors, both of which can be absorbed by multiplying $\{\delta^\pm_n\}$ by an $O(1)$ factor in the regime \eqref{m184}:
\begin{itemize}
\item[$(i)$] On the l.~h.~s.~in Lemma~\ref{L:33}, replacing $\E I_L$ with $\E W(\gamma^*_{\text{loop}})/\len(\gamma^*_{\text{loop}})$ or $\E I_{L/l}$ with $\E W(\gamma^*_{\text{loop},l})/\len(\gamma^*_{\text{loop},l}) $ results in an additional error $\lesssim l^{-1}$ by Lemma \ref{L:9}; Here $\gamma^*_{\text{loop},l}$ is the almost optimal loop in $B_L$ with side-length $\ge l$.
\item[$(ii)$] On the r.~h.~s. of \eqref{c:105} we need to replace $I_L$ with $W(\gamma^*_{\text{loop}})/\len(\gamma^*_{\text{loop}})$. Consequently, in \eqref{c:105} we can apply Lemma \ref{L:30}(i) with 
\begin{align*}
I:=\E I_L+O(\ln^{1/2}L+l^{-1})\overset{\eqref{m184}}{=}\E I_L(1+O(\ln^{-1}l))
\end{align*}
from which we recover Lemma \ref{L:31}. Similar considerations also apply to $\gamma_{\text{loop},l}^*$ in the proof of Lemma \ref{L:32}.
\end{itemize}

\subsection{Proof of Theorem~\ref{T:1}} We will now show that the convergence in Theorem~\ref{T:1} holds. We remark that we are working under the additional hypothesis that the a-priori upper bound in Lemma~\ref{T:7} holds. For the latter, we refer the reader to Section~\ref{s:details}, where it is shown via a continuity argument. 

\medskip

As a first step, let us show the following post-processing of Lemma~\ref{L:34}:
\begin{align}
\mathbb{E}^{4/3}I_L-\mathbb{E}^{4/3}I_{L/l}\approx\textstyle\frac{4}{3}a^*\ln l\quad\mbox{provided \eqref{m184} holds.}\label{m185}
\end{align}
We may rewrite Lemma~\ref{L:34} as
\begin{align*}
\mathbb{E}I_L\big(1-(1+\delta_{L,l})a^*(\mathbb{E}I_L)^{-4/3}\ln l\big)=\mathbb{E}I_{L/l},
\end{align*}
where $\delta_{L,l}$ is some error which vanishes in the regime \eqref{m184}, and may change from line to line. Raising both sides to the power $4/3$ and Taylor expanding, we get
\begin{align*}
(\mathbb{E}I_L)^{4/3}\big(1-\textstyle{\frac{4}{3}}(1+\delta_{L,l})a^*(\mathbb{E}I_L)^{-4/3}\ln l\big)=(\mathbb{E}I_{L/l})^{4/3},
\end{align*}
which is exactly \eqref{m185}.

\medskip

Consider a sequence $\{L_n\}_n$ defined inductively from $l_0=L_0\gg1$ by
\begin{align*}
L_n=l_nL_{n-1}\quad\mbox{with}\quad\ln l_n=[\ln^\alpha L_n]\quad\mbox{for $n\ge1$},
\end{align*}
for some $\alpha\in(0,\frac{1}{8})$, so that \eqref{m184} holds for the couple $(L_n,l_n)$. It is enough to prove the convergence along this sequence, since $L\mapsto\mathbb{E}I_L$ is increasing and $\ln L_{n+1}/(\ln L_n)\to1$. Repeatedly applying \eqref{m185}, and using $\ln L_N=\sum_{n=0}^N\ln l_n$ we get
\begin{align*}
\frac{\mathbb{E}^{4/3}I_{L_N}}{\ln L_N}=\frac{\mathbb{E}^{4/3}I_{L_0}}{\ln L_N}+{\textstyle \frac{4}{3}}a^*\frac{\sum_{n=1}^N(1+\delta_{L_n,l_n})\ln l_n}{\sum_{n=0}^N\ln l_n}.
\end{align*}
Sending $N\to\infty$ finishes the proof.

\subsection{Discrepancy with relative perimeter}\label{SS:rel}

In this subsection, we show how to adjust the previous strategy to prove Theorem~\ref{T:6}. The concentration of $J_L$ is similar to the one of $I_L$ and follows from the Poincar\'e-Sobolev inequality for $BV$ functions\footnote{recall that $\lambda$ is the Lebesgue measure in $\mathbb{R}^2$.}:
\begin{align*}
\mathbb{E}^{1/2}(\frac{W(\gamma,Q_L)}{\len(\gamma,Q_L)})^2\le\frac{1}{\int_{Q_L}|Dw|}\big(\int_{Q_L}\big(w(\gamma,z)-\frac{1}{\lambda(Q_L)}\int_{Q_L}w\big)^2dz\big)^{1/2}\lesssim1.
\end{align*}
Moreover, we note that the lower bound on the expectation $\mathbb{E}J_L$, with the same asymptotic constant as $I_L$, is essentially a direct consequence of Theorem~\ref{T:1}. Indeed, since $\Gamma_{L/2,1}\subset\Gamma'_{L,1}$, we obtain
\begin{align}\label{m186}
J_L\ge\max_{\gamma\in\Gamma_{L/2,1}}\frac{W(\gamma,Q_L)}{\len(\gamma,Q_L)}\ge I_{L/2}-\frac{1}{\lambda(Q_L)}\big|\int_{Q_L}\xi|\max_{\gamma\subset Q_L}\frac{\int dz~w(\gamma,z)}{\len(\gamma)}.
\end{align}
Note that $\int_{Q_L}\xi$ is a Gaussian of variance $L^2$. Together with the $BV$ embedding in $L^1$ applied to $w(\gamma,\cdot)$,
\begin{align*}
\int_{Q_L}dz~w(\gamma,z)\lesssim L\int_{Q_L}|Dw|=L\len(\gamma),
\end{align*}
we obtain that the expectation of the second r.~h.~s.~term in \eqref{m186} is $\lesssim1$. With these observations, we have obtained the following

\begin{lemma}\label{L:36}
\begin{align*}
\liminf_{L\uparrow\infty}\frac{\mathbb{E}J_L}{\ln^{3/4}L}\ge\lim_{L\uparrow\infty}\frac{\mathbb{E}I_L}{\ln^{3/4}L}\quad\mbox{and}\quad\|J_L-\mathbb{E}J_L\|_2\lesssim1.
\end{align*}
\end{lemma}

The remaining part of the proof is devoted to the upper bound, for which we repeat the proof of Theorem~\ref{T:1} with appropriate adjustment. The coarse-graining is now defined as follows: Given a curve $\gamma\in\Gamma'_{L,1}$, one keeps the part of $\gamma$ running over the boundary $\partial Q_L$ unchanged and coarse-grains only $\gamma\cap Q_L$,  again by keeping  one vertex every $l$. This defines $\gamma_{\ge l}$ and a decomposition of the form
\begin{align*}
\gamma=\gamma_{\ge l}+\sum_{n=1}^N\gamma_n.
\end{align*}
Similarly to \eqref{m169}, one gets
\begin{align}\label{m188}
\frac{W(\gamma,Q_L)}{\len(\gamma,Q_L)}&-\frac{W(\gamma_{\ge l},Q_L)}{\len(\gamma_{\ge l},Q_L)}=\frac{1}{\len(\gamma_{\ge l},Q_L)}\times\nonumber\\&\times\sum_{n=1}^NW(\gamma_n)-\frac{W(\gamma,Q_L)}{\len(\gamma,Q_L)}\big(\len(\gamma_n)-l_n\big)-\frac{\int w(\gamma_n,z)}{\lambda(Q_L)}\int_{Q_L}\xi.
\end{align}

\medskip

Similarly to Section \ref{SS:Simple}, in order to control statistical errors, we need to choose $\gamma$ to be an almost maximizer of $J_L$ with additional good topological properties. This is provided by the following analogue of Lemma~\ref{L:9}, the proof of which relies on the regularity theory for the problem $J_L$, and is postponed to the next subsection.

\begin{lemma}\label{L:23}
Assume that $l,L\gg1$ satisfy \eqref{m184}. Then, outside of an event of probability $\lesssim L^{-1}$, there exists a simple curve $\gamma^*_{\text{loop}}\in\Gamma'_{L,1}$ with normal $\nu$ such that
\begin{align*}
\max_{|z-z'|\le l}|\nu(z)-\nu(z')|\lesssim\ln l/\ln^{1/8}L\quad\mbox{and}\quad J_L-\frac{W(\gamma^*_{\text{loop}},Q_L)}{\len(\gamma^*_{\text{loop}},Q_L)}\lesssim l^{-1}.
\end{align*}
Moreover $\gamma^*_{\text{loop}}\cap B_l(z)\cap Q_L$ is a single curve for every $z\in B_L$.
\end{lemma}

With this at hand, we can repeat the proof of the upper bound in Subsection~\ref{SS:upperlower-bounds}. Notice that the last property in the previous lemma is exactly the one needed to obtain independence and control the statistical errors as in Subsection~\ref{SS:statistics}. 

\medskip

The only new error comes from the last term in \eqref{m188}, which will be absorbed into the other errors. Indeed, by Lemma \ref{L:23} and working with $\gamma^*_{\text{loop}}$, we have that $\gamma_n$ are the graphs of functions $h_n$ with $\text{Lip}(h_n)\le\theta\ll1$. Then
\begin{align*}
\frac{1}{\len(\gamma_{\ge l})}\sum_{n=1}^N\int dz~w(\gamma_n,z)\le\frac{1}{\len(\gamma_{\ge l})}\sum_{n=1}^N\int dx|h_n(x)|\lesssim\theta l.
\end{align*}
When multiplied by the Gaussian $\lambda(Q_L)^{-1}\int_{Q_L}\xi$ of order $L^{-1}$, we get an error $\ll l/L$. This is negligible, since the r.~h.~s.~of \eqref{m188} has size $(\mathbb{E}J_L)^{-1/3}\ln l\sim(\ln^{-1/4}L)\ln l$. Hence, we have obtained the following.

\begin{lemma}
Assume that \eqref{m184} holds. Then
\begin{align*}
\mathbb{E}J_L-\mathbb{E}J_{L/l}\lessapprox(\mathbb{E}J_L)^{-1/3}a^*\ln l.
\end{align*}
\end{lemma}

By the same real analysis and continuity argument as before, we finally deduce the upper bound
\begin{align*}
\limsup_{L\uparrow\infty}\frac{\mathbb{E}J_L}{\ln^{3/4}L}\le({\textstyle\frac{4}{3}}a^*)^{3/4}=\lim_{L\uparrow\infty}\frac{\mathbb{E}I_L}{\ln^{3/4}L}.
\end{align*}
Together with the lower bound in Lemma~\ref{L:36}, this concludes the proof of Theorem~\ref{T:6}.

\subsection{Boundary regularity of the optimal curve}\label{SS:bound-reg}

The goal of this subsection is to prove Lemma~\ref{L:23}. We start with the following observation: Any curve $\gamma\in\Gamma'_{L,1}$ can be formally decomposed as
\begin{align*}
\gamma=\sum_{k=1}^N\gamma_k\quad\mbox{with $\gamma_k\in\Gamma'_{L,1}$ entering and exiting $\partial Q_L$ at most once},
\end{align*}
in the sense that the two functionals are additive 
\begin{align*}
W(\gamma,Q_L)=\sum_{k=1}^NW(\gamma_k,Q_L)\quad\mbox{and}\quad\len(\gamma,Q_L)=\sum_{k=1}^N\len(\gamma_k,Q_L).
\end{align*}
Since $(\sum w_k)/(\sum l_k)\le\max{w_k/l_k}$ for every two sequences $\{w_k\}$ and $\{l_k\}$ with $l_k>0$, we deduce that
\begin{align*}
\mbox{there exists a maximizer $\gamma^*$ entering and exiting $\partial Q_L$ at most once.}
\end{align*}

\medskip

The starting point for the regularity theory is again the following fact, which follows as \eqref{m57}: If $\gamma^*$ is optimal for $J_L$, then
\begin{align*}
\gamma^*\text{ minimizes }\len(\gamma,Q_L)-\epsilon W(\gamma,Q_L)\quad \mbox{where}\quad \epsilon:=(J_L)^{-1}
\end{align*}
which, by the concentration and preliminary bound for $J_L$ established in the previous section, it suffices to show regularity of perimeter almost minimizers. The strategy is the same as before: We first show regularity on one single branch, and then upgrade it to a global regularity statement.

\medskip

For $\gamma\in\Gamma'_{L,1}$, let $\mathbb{S}_{in}$ be the connected subset of times that $\gamma$ spends in $Q_L$ and let $V(\gamma)\subset\bar{\mathbb{S}}_{in}$ be the set of vertex-times including the entering and exiting times $t^{ent},t^{ext}$, if $\gamma$ is not entirely contained in $Q_L$. Branches $\gamma_{t,l}$ are defined for $t\in\bar{\mathbb{S}}_{in}$ as in \eqref{c:90}, with the additional requirement that $[t^{in}_{t,l},t^{out}_{t,l}]\subset\bar{\mathbb{S}}_{in}$, so that they are now sub-curves of $\gamma|_{\bar{\mathbb{S}}_{in}}$. With this notation, the proofs of analogues of Lemma~\ref{L:7} (which simply relies on a priori bounds on $J_L$) and of Lemma~\ref{L:12} are the same. Repeating then the same steps up to Corollary~\ref{Cor:1}, we obtain the following: Let  $l\gg 1$  be such that \eqref{m154} holds. Then, outside of an event of probability $ \le L^{-1}$,
\begin{align*}
\mbox{$\gamma^*$ has diameter $\ge l$}
\end{align*} 
and for any $ t\in\mathbb{S}_{in}$ and $s,s'\in [t^{in}_{t,l},t^{out}_{t,l}]$ 
\begin{align*}
|\nu^*(s)-\nu^*(s')|\lesssim\frac{\ln(e+|\gamma^*(s)-\gamma^*(s')|)}{\ln^{1/8}L}.
\end{align*}

\medskip

In the following lemma, we give a description of $\gamma^*$ close to the boundary. To this end, we denote by 
\begin{align*}
\nu^{ext}:=\nu^*(s)\quad\mbox{for any time $s$ with $0<t^{ext}-s<1$,}
\end{align*}
and analogously for $\nu^{ent}$. Moreover, given $z\in\partial Q_L$ let $\nu_{\partial Q_L}(z)$ be the normal to the boundary $\partial Q_L$. 

\begin{lemma}\label{L:28}
If $l$ satisfies \eqref{m154}, then outside of an event of probability $\le L^{-1}$,
\begin{align*}
&(i)&&|\nu^{ent}\cdot\nu_{\partial Q_L}\big(\gamma^*(t^{ent})\big)|,|\nu^{ext}\cdot\nu_{\partial Q_L}\big(\gamma^*(t^{ext})\big)|\lesssim\ln^{-1/8}L\\
&(ii)&&|\gamma^*(t^{ent})-z|,|\gamma^*(t^{ext})-z|\gtrsim l\quad\mbox{for every vertex $z$ of $Q_L$;}\\
&(iii)&&|\gamma^*(t^{ext})-\gamma^*(t^{ent})|\gtrsim l;\\
&(iv)&&\mbox{If}\quad 0<d(\gamma^*(t),\partial Q_L)\ll l,\quad\mbox{then}\quad\mbox{$\gamma^*_{t,l}$ intersects $\partial Q_L.$}
\end{align*}
\end{lemma}

From these statements we deduce that $\gamma^*$ exits and enters the boundary $\partial Q_L$ almost perpendicularly and at two points $\gamma^*(t^{ent}),\gamma^*(t^{ext})$ which are at distance $\gtrsim l$ between each other and from the vertices of $Q_L$. Moreover, we learn from $(iv)$ that $\gamma^*$ remains far from the boundary except for the two branches $\gamma^*_{t,l}$ with $t=t^{ext}$ or $t=t^{ent}$. Thus, we have deduced that
\begin{align*}
\mbox{$\gamma^*\cap Q_L\cap B_l(z)$ consists of a single curve for every $z\in\partial Q_L$.}
\end{align*}
We can now conclude the proof of Lemma~\ref{L:23}. We have to treat only the case when $\gamma^*$ is not simple. If that happens, we would like to pick $\gamma^*_{\text{loop}}$ as a slightly modified subloop of $\gamma^*$, with the same construction as in Lemma~\ref{L:9}. This can indeed be done in view of the last remark, which implies that $\gamma^*_{\text{loop}}$ is at distance $\gtrsim l$ from the boundary.

\section{Detailed proofs}\label{s:details}

In the following, we will need the following notation. Given two, possibly open, Lipschitz curves $\gamma_1:[0,T_1]\to\mathbb{R}^2$ and $\gamma_2:[0,T_2]\to\mathbb{R}^2$ such that $\gamma_1(T_1)=\gamma_2(0)$, we define their concatenation $\gamma_1\ast\gamma_2:[0,T_1+T_2]\to\mathbb{R}^2$ by
\begin{align}\label{m86}
\gamma_1\ast\gamma_2(t):=\begin{cases}
\gamma_1(t)&\text{if}~t\in[0,T_1],\\
\gamma_2(t-T_1)&\text{if}~t\in[T_1,T_1+T_2].
\end{cases}
\end{align}
It will also be convenient to write
\begin{align}\label{m85}
-\gamma:=\mbox{the curve obtained from $\gamma$ by changing the orientation.}
\end{align}

{\sc Proof of Lemma~\ref{L:12}.} 
We prove the lower bound on the diameter of $\gamma^*$ arguing by contradiction. Notice that if $\gamma^*$ had diameter $<2l$, then it would be contained in some ball $B_l(z)\subset B_L$ and therefore $ I_l(z)= I_L$. This is in contradiction with \eqref{c:24}. As a consequence we also see that $t^{in}_{t,l}$ and $t^{out}_{t,l}$ are well-defined.

\medskip

Let us introduce the open curve $\tilde\gamma_{t,l}$ defined by
\begin{align*}
\tilde\gamma_{t,l}\overset{\eqref{m86}}{:=}\begin{cases}
    [\gamma^*(t^{in}_{t,l}),\gamma^*(t^{out}_{t,l})] & \text{ if } |\gamma^*(t^{in}_{t,l})-\gamma^*(t^{out}_{t,l})|\geq 1,\\
    [\gamma^*(t^{in}_{t,l}),z]*[z,\gamma^*(t^{out}_{t,l})] & \text{ if }|\gamma^*(t^{in}_{t,l})-\gamma^*(t^{out}_{t,l})|<1,
\end{cases}
\end{align*}
where $z\in B_l(\gamma^*(t))$ is a point such that $|\gamma^*(t^{in}_{t,l})-z|=|\gamma^*(t^{out}_{t,l})-z|=1$. Such a point $z$ exists for $l\ge3$. For later purposes, notice that
\begin{align}\label{m21}
\len(\tilde\gamma_{t,l})\le\max\{|\gamma^*(t^{in}_{t,l})-\gamma^*(t^{out}_{t,l})|,2\}.
\end{align}
Moreover, let $\tilde\gamma\in\Gamma_{L,1}$ be the curve obtained from $\gamma^*$ by replacing $\gamma^*_{t,l}$ with $\tilde\gamma_{t,l}$.

\medskip

Let us show that 
\begin{align}\label{c:3}
\len(\gamma_{t,l}^*)(1-\sup_{z\in B_L} I_l(z)/ I_L)\leq \len(\tilde\gamma_{t,l})(1+\sup_{z\in B_L} I_l(z)/ I_L).
\end{align}
Indeed, using $\tilde\gamma$ as a competitor for the problem in \eqref{m57}, we have
\begin{align}\label{m84}
\len(\gamma^*)-\epsilon W(\gamma^*)\le\len(\tilde\gamma)-\epsilon W(\tilde\gamma)\quad\mbox{with}\quad\epsilon:= I_L^{-1}.
\end{align}
Notice that with respect to the difference of the lengths, we have
\begin{align*}
\len(\gamma^*)-\len(\tilde\gamma)=\len(\gamma^*_{t,l})-\len(\tilde\gamma_{t,l}),
\end{align*}
while the difference of the field terms can be written as
\begin{align*}
 W(\gamma^*)- W(\tilde\gamma)= W(\bar\gamma)\quad\mbox{with}\quad\bar\gamma\overset{\eqref{m86},\eqref{m85}}{:=}\gamma^*_{t,l}\ast(-\tilde\gamma_{t,l}).
\end{align*}
By construction\footnote{here $\Gamma_l(z)$ denotes the translated set of curves $z+\Gamma_{l,1}.$} $\bar\gamma\in\Gamma_{l,1}(\gamma^*(t))$, thus the r.~h.~s.~ above is estimated by
\begin{align*}
 W(\bar\gamma)\le I_l(\gamma^*(t))\len(\bar\gamma)\le\sup_{z\in B_L} I_l(z)\big(\len(\gamma^*_{t,l})+\len(\tilde\gamma_{t,l})\big).
\end{align*}
Substituting the last three equations into \eqref{m84} and rearranging the terms, one obtains \eqref{c:3}.
\medskip

By \eqref{c:3} and that $\frac{1+x}{1-x}\leq 1+3x$ for $0<x\ll 1$ and the assumption \eqref{c:24}, we deduce
\begin{align*}
&\len(\gamma^*_{t,l})\leq (1+\eta_l)\len(\tilde\gamma_{t,l
})\le(1+\eta_l)\max\{|\gamma^*(t^{in}_{t,l})-\gamma^*(t^{out}_{t,l})|,2\}\\
&\text{for}\quad\eta_l\lesssim \sup_{z\in B_L} I_l(z)/ I_L.
\end{align*}
Finally, since by definition $\len(\gamma^*_{t,l})\ge2l-2$ and $2l-2>2(1+\eta_l)$ for $l\ge3$, we can replace the maximum on the r.~h.~s.~with $|\gamma^*(t^{in}_{t,l})-\gamma^*(t^{out}_{t,l})|$ and conclude.
\qed

\medskip

{\sc Proof of Lemma \ref{L:7}. }
We define the event
\begin{align*}
A_M&:=\{ \sup_{z\in B_L} I_l(z)\leq C\ln^{3/4}l+C\ln^{1/2}L+M,\quad  I_L\geq {\textstyle{\frac{1}{2}}}\E I_L \}
\end{align*}
where $C$ is the hidden universal constant in the upper bound of Lemma~\ref{T:7}, and show that
\begin{align}\label{c:23}
\ln\PP(A_M^c)\lesssim -M^2\quad\mbox{provided}~M\ll\ln^{3/4}L.
\end{align}
This finishes the proof by the choice $M=C'\ln^{1/2}L$ for some constant $C'$ such that $\PP(A_M^c)\leq L^{-1} $.
Here comes the proof. It follows from the union bound that 
\begin{align}\label{c:25}
\PP(A^c_M)\leq&\PP(\sup_{z\in B_L} I_{l}(z)\geq C\ln^{3/4}l+C\ln^{1/2}L+M)\nonumber\\&+\PP( I_L\leq {\textstyle{\frac{1}{2}}}\E I_L).
\end{align}

For the first term on r.~h.~s.~, we have 
\begin{align*}
\sup_{z\in B_L} I_l(z)\leq \sup_{\bar z\in l\mathbb{Z}^2\cap B_L}\sup_{z\in B_{l}(\bar z)} I_l(z)\leq \sup_{\bar z\in l\mathbb{Z}^2\cap B_L} I_{2l}(\bar z),
\end{align*}
which, since $\mathbb{E}I_{2l,1}(\bar z)$ does not depend on $\bar z$, can be decomposed as
\begin{align*}
\sup_{\bar z\in l\mathbb{Z}^2\cap B_L} I_{2l}(\bar z)=\mathbb{E}I_{2l}+\sup_{\bar z\in l\mathbb{Z}^2\cap B_L} I_{2l}(\bar z)-\E I_{2l}(\bar z).
\end{align*}
Using $\#(l\mathbb{Z}^2\cap B_L)\lesssim(L/l)^{2}$, \eqref{m83} and the concentration from Lemma~\ref{T:7}, we obtain
\begin{align*}
\sup_{\bar z\in l\mathbb{Z}^2\cap B_L} I_{2l}(\bar z)-\E I_{2l}(\bar z)\lesssim
\ln^{1/2} (L/l)+\bar I\quad\mbox{with}\quad\|\bar I\|_2\le1.
\end{align*}
Because of Lemma~\ref{app:4}(a), this implies
\begin{align*}
\ln\PP(\sup_{z\in B_L} I_{l}(z)\gtrsim\ln^{3/4}l+\ln^{1/2}L+M)\lesssim-M^2.
\end{align*}
Finally, for the second term in \eqref{c:25},
\begin{align*}
\ln\PP( I_L\leq{\textstyle{\frac{1}{2}}}\E I_L)\overset{\text{Lemma~\ref{T:7}}}&{\le}\ln\PP(| I_L-\E I_L|\gtrsim\ln^{3/4}L)\lesssim-\ln^{3/2}L.
\end{align*}
This finishes the proof of \eqref{c:23}.
\qed

\medskip

{\sc Proof of Lemma~\ref{L:8}.} 
We fix a length scale $ 3\leq l'\leq l $ and a time $t\in\mathbb{S}$ and consider  the branch $ \gamma_{t,l'} $. Let $ \gamma^{line}_{t,l'}:=[\gamma(t^{in}_{t,l'}),\gamma(t^{out}_{t,l'})] $, the normal of which we denote by $ \nu_{t,l'} $. They are well defined by the assumption that $\gamma$ has diameter $\geq 2l$ and that $l'\geq 3$. We define the excess at time $ t $ and scale $ l' $ to be 
\begin{align*}
\Exc(t,l'):=\frac{1}{l'}\int_{t^{in}_{t,l'}}^{t^{out}_{t,l'}}|\nu(s)-\nu_{t,l'}|^2 ds.
\end{align*}
We also denote by $\tau(\cdot)$ and $\tau_{t,l'}$  the tangent vectors of $\gamma$ and $\gamma_{t,l'}^{line}$, respectively.
\medskip

We first show the following property
\begin{align}\label{c:4}
\Exc(t,l') \lesssim \eta,\quad \text{for }3\leq l'\leq l\text{ and }t\in\mathbb{S}.
\end{align}
Starting from the equality
\begin{align*}
\int_{t^{in}_{t,l'}}^{t^{out}_{t,l'}}\nu(s)\cdot \nu_{t,l'}ds&=\int_{t^{in}_{t,l'}}^{t^{out}_{t,l'}}\tau(s)\cdot\tau_{t,l'}ds\\&=(\gamma(t^{out}_{t,l'})-\gamma(t^{in}_{t,l'}))\cdot\tau_{t,l'}=|\gamma(t^{in}_{t,l'})-\gamma(t^{out}_{t,l'})|
\end{align*}
and expanding the square in the definition of the excess, we have
\begin{align*}
\Exc(t,l')=2 (\len(\gamma_{t,l'})-\len(\gamma^{line}_{t,l'}))/l'\overset{\eqref{m54}}{\lesssim}\eta.
\end{align*}

\medskip

Next, we use the control on the excess to estimate the change in the average normal $\nu_{t,l'}-\nu_{t,l'/2}$ between two adjacent dyadic scales,
\begin{align}\label{c:26}
|\nu_{t,l'}-\nu_{t,l'/2}|&\lesssim \sqrt{\eta}\quad \text{for }1\ll l'\leq l.
\end{align}
The estimate follows from
\begin{align*}
|\nu_{t,l'}-\nu_{t,l'/2}|^2&=\frac{1}{\len(\gamma_{t,l'/2})}\int_{t^{in}_{t,l'/2}}^{t^{out}_{t,l'/2}}|\nu_{t,l'}-\nu_{t,l'/2}|^2ds\\&\lesssim \frac{1}{l'}\int_{t^{in}_{t,l'/2}}^{t^{out}_{t,l'/2}}|\nu(s)-\nu_{t,l'/2}|^2ds+\frac{1}{l'}\int_{t^{in}_{t,l'}}^{t^{out}_{t,l'}}|\nu(s)-\nu_{t,l'}|^2ds\\&\lesssim \Exc(t,l'/2)+\Exc(t,l')\lesssim\eta.
\end{align*}
Exploiting the small-scale cutoff in the configuration space, we complement the previous estimate with a second one, which will be used for $l'$ of order 1 
\begin{align}
\label{c:27}
|\nu(t)-\nu_{t,l'}|&\lesssim \sqrt{\eta l'}\quad \text{for }1\leq l'\leq l.
\end{align}
Indeed, let us denote by $t^-,t^+$ the two consecutive vertex-times such that $t^-<t\leq t^+$ (namely, $[\gamma(t^-),\gamma(t^+)]$ is the edge on which $\gamma(t)$ lies). Then 
\begin{align*}
|\nu(t)-\nu_{t,l'}|^2=\frac{1}{|t^+-t^-|}\int_{t^-}^{t^+}|\nu-\nu_{t,l'}|^2\lesssim \int_{t^{in}_{t,l'}}^{t^{out}_{t,l'}}|\nu(s)-\nu_{t,l'}|^2ds\overset{\eqref{c:4}}{\lesssim} \eta l'.
\end{align*}
Notice that with the same proof of \eqref{c:26}, one also obtains that for any $t,s\in\mathbb{S}$
\begin{align}\label{m15}
\mbox{if $\gamma_{s,l'/2}$ is contained in $\gamma_{t,l'}$,}\quad\mbox{then}\quad|\nu_{t,l'}-\nu_{s,l'/2}|\lesssim\sqrt{\eta}.
\end{align}

\medskip

As a consequence of \eqref{c:26} \& \eqref{c:27}, let us now show that
\begin{align}\label{c:28}
|\nu(t)-\nu_{t,l'}|\lesssim\sqrt{\eta}\ln(e+l')\quad\mbox{for}~1\le l'\le l.
\end{align}
We show the above for $l'\gg1$, since the case when $l'\lesssim 1$ is covered by \eqref{c:27}. Let $l_0$ be the universal constant such that \eqref{c:26} holds for $l_0\le l'\le l$ and let $\bar k$ be such that $l_0\in [2^{-\bar k-1}l',2^{-\bar k}l']$. Then
\begin{align*}
|\nu(t)-\nu_{t,l'}|&\le|\nu(t)-\nu_{t,2^{-\bar k}l'}|+\sum_{k=0}^{\bar k-1}|\nu_{t,2^{-k}l'}-\nu_{t,2^{-(k+1)}l'}|\\\overset{\eqref{c:26},\eqref{c:27}}&{\lesssim}\sqrt{\eta}(l_0+\bar k)\lesssim\sqrt{\eta}\ln(e+l').
\end{align*}

\medskip

Let us now conclude. Let $s,s'\in\mathbb{S}$ be such that
\begin{align*}
\sqrt{\eta}\ln(e+|s-s'|)\ll1,
\end{align*}
otherwise the claim is trivial. Consider the branch $\gamma_{s,l}$ with
\begin{align*}
l=C(|s-s'|+1)\quad\mbox{for some universal $C\gg1$ so that}\quad\gamma_{s',l/2}\subset\gamma_{s,l}.
\end{align*}
Then by the triangle inequality
\begin{align*}
|\nu(s)-\nu(s')|&\le|\nu(s)-\nu_{s,l}|+|\nu(s')-\nu_{s',l/2}|+|\nu_{s,l}-\nu_{s',l/2}|\\
\overset{\eqref{m15},\eqref{c:28}}&{\lesssim}\sqrt\eta\ln l\lesssim\sqrt\eta\ln(e+|s-s'|).
\end{align*}
\qed

\medskip

{\sc Proof of Remark~\ref{rk:1}.}
Let us first show \eqref{c:47}. The assumption implies that for any $s\in[0,T]$,
\begin{align*}
 |\nu(s)-\frac{|\gamma(0)-\gamma(T)|}{T}\nu^{line}|&=|\nu(s)-\frac{1}{T}\int_{0}^{T}\nu(t)dt|\\&\le\frac{1}{T}\int_0^T|\nu(s)-\nu(t)|dt\le\sqrt{\eta}
\end{align*}
where we used the fact that $\int^{T}_{0}\nu(t)dt=|\gamma(0)-\gamma(T)|\nu^{line}$. Thus 
\begin{align*}
|1-\frac{|\gamma(0)-\gamma(T)|}{T}|\leq |\nu(s)-\frac{|\gamma(0)-\gamma(T)|}{T}\nu^{line} |\leq\sqrt{\eta}.
\end{align*}
Now \eqref{c:47} follows by combining the above two inequalities,
\begin{align*}
|\nu(s)-\nu^{line}|\leq |\nu(s)-\frac{|\gamma(0)-\gamma(T)|}{T}\nu^{line}|+|1-\frac{|\gamma(0)-\gamma(T)|}{T}|\leq 2\sqrt{\eta}.
\end{align*}
\medskip

Finally, to show \eqref{c:45}, note that
\begin{align*}
2(\len(\gamma)-|\gamma(0)-\gamma(T)|)=\int_{0}^{T}|\nu(t)-\nu^{line}|^2dt\overset{\eqref{c:47}}{\leq} 4\eta\len(\gamma),
\end{align*}
where the first equality follows from expanding the square. Rearranging and using that $\frac{1}{1-x}\leq 1+2x$ for $0\leq x\ll 1$, it leads to \eqref{c:45}.
\qed

\medskip

{\sc Proof of Lemma~\ref{L:18}.}
We will show below that if $l'\gg1$, then with probability $\ge1-L^{-1}$ for any $z,z'$ in the image of $\gamma^*$,
\begin{align}\label{e3}
|\nu^*(z)-\nu^*(z')|\lesssim\frac{\ln l'}{\ln^{1/8}L}+\frac{|z-z'|}{l'}.
\end{align}
Now we fix $l$, assume that $|z-z'|\leq l$ and apply the above estimate with $l'=(l\ln^{1/8}L)/\ln l$,
\begin{align*}
|\nu^*(z)-\nu^*(z')|\lesssim\frac{\ln l'}{\ln^{1/8}L}+\frac{\ln l}{\ln^{1/8}L}\lesssim\frac{\ln l}{\ln^{1/8}L},
\end{align*}
where the last inequality holds provided $\ln l\gg\ln\ln L$ (which implies $\ln l'\lesssim\ln l$). This finishes the proof of \eqref{e3}.

\medskip

Here comes the proof of \eqref{e3}, where for simplicity we use $l$ in place of $l'$. Let us denote by 
\begin{align*}
\mbox{$\gamma^*_{s,l}:=$ the branch of $\gamma^*$ contained in $B_l(z)$ and starting from $s$.}
\end{align*}
Let $s_{in},s_{out}$ be its entering and exiting times and let $\nu_{s,l}$ be normal to the segment $[\gamma^*(s_{in}),\gamma^*(s_{out})]$. Let us define analogously $s'_{in},s'_{out}$ and $\nu_{s',l}$. 

\medskip

In view of Corollary~\ref{Cor:1}, we have that 
\begin{align*}
|\nu^*(s)-\nu_{s,l}|,|\nu^*(s')-\nu_{s',l}|\lesssim\ln l/\ln^{1/8}L,
\end{align*}
so that we have only to control $|\nu_{s,l}-\nu_{s',l}|$. If the segments $[\gamma^*(s_{in}),\gamma^*(s_{out})]$ and $[\gamma^*(s'_{in}),\gamma^*(s'_{out})]$ do not intersect, then we have
\begin{align*}
|\nu_{s,l}-\nu_{s',l}|\lesssim|\gamma^*(s)-\gamma^*(s')|/l.
\end{align*}
If otherwise they intersect, we construct a competitor in the following way: We replace $\gamma^*|_{[s_{in},s_{out}]}$ and $\gamma^*|_{[s'_{in},s'_{out}]}$ by the segments $[\gamma^*(s_{in}),\gamma^*(s'_{out})]$ and $[\gamma^*(s'_{in}),\gamma^*(s_{out})]$, respectively (see Figure~\ref{f2}). With the same reasoning given in Lemma~\ref{L:12}, $\gamma^*$ has almost minimal length in $B_l(z)$, to the effect that
\begin{align*}
&\len(\gamma^*_{s,l})+\len(\gamma^*_{s',l})\le(1+\eta)\big(|\gamma^*(s_{in})-\gamma^*(s'_{out})|+|\gamma^*(s'_{in})-\gamma^*(s_{out})|\big)\\
&\mbox{with}\quad\eta\lesssim\ln l/\ln^{1/8}L.
\end{align*}
Combining it with the trivial lower bound
\begin{align*}
|\gamma^*(s_{in})-\gamma^*(s_{out})|+|\gamma^*(s'_{in})-\gamma^*(s'_{out})|\le\len(\gamma^*_{s,l})+\len(\gamma^*_{s',l}),
\end{align*}
it yields the following constraint
\begin{align*}
&|\gamma^*(s_{in})-\gamma^*(s_{out})|+|\gamma^*(s'_{in})-\gamma^*(s'_{out})|\\&
\le(1+\eta)\big(|\gamma^*(s_{in})-\gamma^*(s'_{out})|+|\gamma^*(s'_{in})-\gamma^*(s_{out})|\big).
\end{align*}
We conclude by elementary geometry that
\begin{align*}
|\nu_{s,l}-\nu_{s',l}|\lesssim\sqrt{\eta}+|\gamma^*(s)-\gamma^*(s')|/l.
\end{align*}
\qed

\medskip

{\sc Proof of Lemma~\ref{L:30}.} The strategy is to reduce both estimates to problems with fixed boundary conditions and then apply Theorem~\ref{T:3}. To this end, we introduce the four additional points $z_i^\pm$ for $i=1,2$ such that (Figure~\ref{f5})
\begin{align*}
&\mbox{$z_1^+, \bar z_1, z_1^-, z_2^-, \bar z_2,z_2^+$ lie on the same line (in this order)}\\&\mbox{and}\quad|\bar z_1-z_1^\pm|=|\bar z_2-z_2^\pm|=2\kappa\bar l.
\end{align*}
We note that, since $\kappa\leq\ln^{-1}\bar l$, we can replace at any time $|z_1^\pm-z_2^\pm|$ and $|z_1-z_2|$ with the reference length $\bar l=|\bar z_1-\bar z_2|$.

\medskip

\emph{Proof of $(i)$.} Starting from a curve $\gamma$ with endpoints $z_1,z_2$, we concatenate it with the segments $[z_1^+,z_1]$ and $[z_2,z_2^+]$ to obtain a curve $\gamma^+$. In view of the oscillations $\le\theta$ of the normal, this is the graph of a function $h^+$ over the segment $[z_1^+,z_2^+]$ and we have the estimates
\begin{align*}
\len(\gamma)-|z_1-z_2|&=\len(\gamma^+)-|z_1^+-z_2^+|+O(\kappa\bar l)\\
&=\big(1+O(\theta^2)\big)D(h^+)+O(\kappa\bar l).
\end{align*}
Let us introduce the curve $\gamma'$ connecting the points $z_1^+,z_1,z_2,z_2^+$, so that
\begin{align*}
W(\gamma)=W(\gamma^+)-W(\gamma')=W(h^+)-W(\gamma').
\end{align*}
In order to estimate the error $W(\gamma')$, we note that $\gamma'$ is a curve through four points contained in a fixed rectangle of side-lengths $\lesssim \bar l$ and $\kappa\bar l$. Thus, by scaling $(x,y)=\big(\bar l\hat x,\kappa\bar l\hat y\big)$ and Lemma~\ref{app:1},
\begin{align*}
\big\|\sup_{z_1,z_2}W(\gamma')\big\|_2\lesssim\kappa^{1/2}\bar l.
\end{align*}

\medskip

Combining the above estimates and applying Theorem~\ref{T:3} to the function $h^+$, we obtain that for some $\|\delta^+\|_{3/2}\lesssim1$
\begin{align*}
&W(\gamma)-I\big(\len(\gamma)-|z_1-z_2|\big)\\&\le W(h^+)-I\big(1+O(\theta^2)\big)D(h^+)+O(I\kappa\bar l)+\delta^+\kappa^{1/2}\bar l\\
&\le\big(1+\delta^+(\ln^{-1}\bar l+\theta^2)\big)a^*I^{-1/3}\bar l\ln\bar l+O(I\kappa\bar l)+\delta^+\kappa^{1/2}\bar l.
\end{align*}
Using that $\kappa\leq I^{-4/3}$, we can combine all the errors together and conclude.

\medskip

\emph{Proof of $(ii)$.} Let us introduce the function $h^-$ over the segment $[z_1^-,z_2^-]$ maximizing the action $W(h)-ID(h)$ among function $h$ which are piecewise linear on intervals of length $\ge 1$ and have vanishing boundary conditions. Then, given any pair of points $z_1,z_2$, the curve $\gamma$ is constructed by concatenating the graph $\gamma^-$ of $h^-$ with the segments $[z_1,z_1^-]$ and $[z_2^-,z_2]$. 

\medskip

From Theorem~\ref{T:3}, we have the lower bound for some $\|\delta^-\|_{3/2}\lesssim1$,
\begin{align*}
W(h^-)-ID(h^-)\ge\big(1-\delta^-\ln^{-1}\bar l\big)a^*I^{-1/3}\bar l\ln\bar l.
\end{align*}
We can estimate again the increment of length and the field term as
\begin{align*}
&\len(\gamma)-|z_1-z_2|=\len(\gamma^-)-|z_1^--z_2^-|+O(\kappa\bar l)\le D(h^-)+O(\kappa\bar l),\\
&W(\gamma)-W(h^-)=W(\gamma')\quad\mbox{with}\quad\big\|\sup_{z_1,z_2}W(\gamma')\big\|_2\lesssim\kappa^{1/2}\bar l.
\end{align*}
Combining these estimates as in point $(i)$ suffices to conclude.\qed

\medskip

{\sc Proof of Lemma~\ref{L:27}.}
Up to translation, rotation and dilation, assume that $z=\gamma(s)=0,\tau(s)=(1,0),l=1$. As a consequence of the hypothesis, $\{\mbox{the image of $\gamma$}\}\cap B_1$ contains at least the graph of a (almost flat) function $h$, which is approximately the diameter $[-1,1]\times\{0\}$ of $B_1$. If $z_1,z_2$ are the entry and exit points of this portion of $\gamma$ into and from $B_1$,  we then have
\begin{align*}
|z_1-(-1,0)|,|z_2-(1,0)|\ll1.
\end{align*}
Assume by contradiction that $\gamma$ enters again $B_1$ and call $z_3$ and $z_4$ the entry and exit points, respectively. From the assumption and \eqref{c:47}, 
\begin{align}\label{m103}
|\frac{z_4-z_3}{|z_4-z_3|}-\frac{z_2-z_1}{|z_2-z_1|}|\ll1\quad\mbox{so that}\quad|\frac{z_4-z_3}{|z_4-z_3|}-(1,0)|\ll1.
\end{align}
By symmetry, we assume that $z_3,z_4$ are in the upper half of $B_1$ (with respect to the graph of $h$). Applying Lemma~\ref{app:3}, we can pick $z_3$ to be the first entry point after $z_2$, going counter-clockwise around $\partial B_1$. Hence the points $z_1,z_2,z_3,z_4$ lie on $\partial B_1$ in this order. Thus the first component $v_1$ of
\begin{align*}
v=(v_1,v_2):=\frac{z_4-z_3}{|z_4-z_3|}\quad\mbox{satisfies}\quad-1\le v_1\ll1
\end{align*}
which is in contradiction with the second inequality in \eqref{m103}.
\qed

\medskip

{\sc Proof of Lemma~\ref{L:9}.} We start by considering  the simple closed curve obtained by $ \gamma^* $  crossing itself for the first time. More precisely, denoting by $0\in\mathbb{S}$ an arbitrary point, let
\begin{align*}
t_1:=\sup\{t~|~\gamma^*:[0,t]\to B_L\text{ is injective} \},
\end{align*}
and let $ t_0\in[0,t_1)$ be the time where $ \gamma^*(t_0)=\gamma^*(t_1) $. By definition $ \gamma^*|_{[t_0,t_1]} $ is a simple, closed, piecewise-linear curve. Note that in general, $ \gamma^*|_{[t_0,t_1]}\notin \Gamma_{L,1} $,  since it might contain two edges with length less than 1, hence we need to slightly modify it.

\medskip

Applying the estimate \eqref{e3} with $l=5$, we learn that $\gamma^*$ self-intersects in two almost-parallel branches,
\begin{align}\label{m156}
|\nu^*(t_0)-\nu^*(t_1)|\lesssim\ln^{-1/8}L\quad\mbox{with probability $\ge1-L^{-1}$}.
\end{align}
Using this information, we construct $\gamma_{\text{loop}}^*\in\Gamma_{L,1}$ as follows: let
\begin{align*}
\mbox{$V(\gamma^*)\ni t_0^+:=$ first vertex-time after $t_0$ such that $|\gamma^*(t_0)-\gamma^*(t_0^+)|\ge1$}
\end{align*}
and analogously $t_0^-,t_1^+,t_1^-$ (see Figure~\ref{f3}). We define $\gamma^*_{\text{loop}}$ to be the concatenation of $ \gamma^*|_{[t_0^+,t_1^-]} $ and the segment $ [\gamma^*(t_1^-),\gamma^*(t_0^+)] $. With this construction, we have that $\gamma^*_{\text{loop}}\in\Gamma_{L,1}$, with the desired regularity. 

\medskip

Finally, we show that $\gamma^*_{\text{loop}}$ almost maximizes the ratio $ I_L$. To this end, we introduce two additional curves $\gamma',\hat\gamma$ satisfying the following properties:
\begin{align*}
&\gamma'\in\Gamma_{L,1}\quad\mbox{while}\quad\mbox{$\hat\gamma$ has $\le10$ vertices and is contained in $B_5(t_0)$};\\
&\gamma^*=\gamma^*_{\text{loop}}+\gamma'+\hat\gamma\quad\mbox{so that}\quad W(\gamma^*)= W(\gamma^*_{\text{loop}})+ W(\gamma')+W(\hat\gamma);\\
& \len(\gamma^*)\geq \len(\gamma_{\text{loop}})+\len(\gamma').
\end{align*}
The precise definition of the two curves is the following:
\begin{align*}
&\gamma':=\gamma^*|_{[t_1^+,t_0^-]}\ast [\gamma^*(t_0^-),\gamma^*(t_1^+)]\quad\mbox{and}\\
&\hat \gamma:=\gamma^*|_{[t_0^-,t_0^+]}*[\gamma^*(t_0^+),\gamma^*(t_1^-)]*\gamma^*|_{[t_1^-,t_1^+]}*[\gamma^*(t_1^+),\gamma^*(t_0^-)].
\end{align*}

\medskip

The previous properties and the optimality of $ \gamma^* $ yield
\begin{align*}
\frac{ W(\gamma')}{\len(\gamma')}\leq \frac{ W(\gamma^*)}{\len(\gamma^*)}\le\frac{ W(\gamma')+W(\gamma^{*}_{\text{loop}})+ W(\hat\gamma)}{\len(\gamma')+\len(\gamma^*_{\text{loop}})},
\end{align*}
which implies (using that $\frac{a}{c}\leq \frac{a+b}{c+d}\Rightarrow \frac{a+b}{c+d}\leq \frac{b}{d}$ with $a= W(\gamma'),c=\len(\gamma')$)
\begin{align*}
I_L=\frac{ W(\gamma^*)}{\len(\gamma^*)}\leq \frac{ W(\gamma^*_{\text{loop}})+W(\hat\gamma)}{\len(\gamma^*_{\text{loop}})}\le\frac{W(\gamma^*_{\text{loop}})}{\len(\gamma^*_{\text{loop}})}+\frac{W(\hat\gamma)}{l};
\end{align*}
the last estimate follows from the regularity of $\gamma^*_{\text{loop}}$ up to scale $l$ which holds  with probability $\ge1-L^{-1}$ provided (cf. Lemma \ref{L:18})
\begin{align*}
\ln l\gg\ln \ln L\quad\mbox{and}\quad\ln l\ll \ln ^{1/8}L.
\end{align*}
To conclude we have to estimate $W(\hat\gamma)$. Since $\hat\gamma$ has at most 10 vertices, applying Lemma~\ref{app:1} we have that $W(\hat\gamma)$ is bounded by an order 1 Gaussian, when $\hat\gamma$ is constrained to lie in a fixed ball with size of order 1. Covering $B_L$ by $L^2$ many balls and taking the supremum, yields (cf.~\eqref{m83})
\begin{align*}
W(\hat\gamma)\lesssim\ln^{1/2}L\quad\mbox{with probability $\ge1-L^{-1}$}.
\end{align*}
Since $\ln l\gg\ln\ln L$, one may absorb the factor $\ln^{1/2}L$ into the definition of $l$ and conclude.
\qed

\medskip

{\sc Proof of Lemma~\ref{T:7} (Upper bound).} We assume by contradiction that there exists a sequence $L_n\to\infty$ such that
\begin{align}\label{m192}
\lim_{n\to\infty}\frac{\mathbb{E}^{4/3}I_{L_n}}{\ln L_n}=\infty.
\end{align}
We can further assume that it satisfies the following property (with $L=L_n$)
\begin{align}\label{m190}
\frac{\E^{4/3}I_L}{\ln L}\ge\frac{\E^{4/3}I_l}{\ln l}\quad\mbox{for every}\quad l\in[2,L-1].
\end{align}
Indeed, one may define for $n\gg1$
\begin{align*}
\tilde L_n:=\sup\{L>0~|~\mathbb{E}^{4/3}I_L<n\ln L\}
\end{align*}
and pick any $L_n\in[\tilde L_n,\tilde L_n+1)$ satisfying $\mathbb{E}^{4/3}I_{L_n}\ge n\ln L_n$.

\medskip

The aim is to rerun the proof, without the a-priori estimate
\begin{align}\label{m191}
\Lambda:=(\mathbb{E}^{4/3}I_L)/\ln L\lesssim1,
\end{align}
but with the additional assumption \eqref{m190}. The upper bound \eqref{m191} affects the proof in three places, which we now analyse: the stochastic estimate in Lemma~\ref{L:7}, which is used as an input to deduce the regularity theory in Lemma~\ref{L:18}; the choice of $\kappa\le I^{-4/3}$ in \eqref{c:100}, which then enters the estimate of the statistical errors in \eqref{m193}; the contribution of the bad event in \eqref{m194}.

\medskip

Let us start by observing that under the assumption \eqref{m190}, the estimates in Lemma~\ref{L:7} and consequently in regularity theory, are free of a-priori upper bounds of $\E I_L$. Namely, with probability $\geq 1-L^{-1}$,
\begin{align*}
\sup_{z\in B_L}I_l(z)/I_L\leq C \frac{\ln^{3/4}l+\ln^{1/2}(L/l)}{\ln^{3/4}L} \mbox{ where $C$ is an absolute constant.}
\end{align*}
Indeed, the potentially uncontrolled term $\mathbb{E}I_l$ appearing in the estimate of the numerator satisfies
\begin{align*}
(\E^{4/3}I_l)/\ln l\le(\E^{4/3}I_L)/\ln L=\Lambda,
\end{align*}
and therefore it is compensated by the denominator. Hence, Lemma~\ref{L:18} still holds, with all the hidden constants independent of $\Lambda$. We note that, since the construction of the simple loop in Lemma~\ref{L:9} relies just on regularity, it is also not affected by $\Lambda$.

\medskip

Let us come to the choice of $\kappa$. In view of Lemma~\ref{L:30}, we are forced to choose $\kappa\sim(\mathbb{E}I_L)^{-4/3}=(\Lambda\ln L)^{-1}$. As an outcome of Subsection~\ref{SS:statistics}, the relative error in \eqref{m180} is estimated by the l.~h.~s.~of \eqref{m192},
\begin{align*}
(\ln^{-1}l+\theta^2)\ln^{2/3}\kappa^{-1}\sim(\ln^{-1}l+\theta^2)(\ln^{2/3}\ln L+\ln^{2/3}\Lambda).
\end{align*}

\medskip

Starting from Lemma~\ref{L:31}, we take the expected values. Note that now $\Lambda$ enters the estimate \eqref{m194} of the contribution of the bad event. Thus,
\begin{align*}
&\mathbb{E}I_L\big(1+O\big(\mathbb{P}^{1/2}(A^c)\big)\big)-\mathbb{E}I_{L/l}\\&\le\big(1+O\big((\ln^{-1}l+\theta^2)(\ln^{2/3}\ln L+\ln^{2/3}\Lambda)\big)\big)a^*(\mathbb{E}I_L)^{-1/3}\ln l+O(l^{-1}),
\end{align*}
where the hidden constants in the $O(\cdot)$ are independent of $\Lambda$ and $O(l^{-1})$ comes from the almost maximality of the simple loop $\gamma^*_{\text{loop}}$. Let us recall that $\theta=\ln l/\ln^{1/8}L$ and let us fix $\ln l=\ln^\alpha L$ for some $\alpha<\frac{1}{8}$, so that
\begin{align*}
(\ln^{-1}l+\theta^2)\ln^{2/3}\ln L\to0.
\end{align*}
Then putting only $\mathbb{E}I_{L/l}$ on the r.~h.~s.~, raising to the $\frac{4}{3}$ and using that $(1+x)^{4/3}\ge1+\frac{4}{3}x$, we get for $L\gg1$
\begin{align*}
&(\mathbb{E}I_L)^{4/3}\big(1+O(l^{-1})+\textstyle{\frac{4}{3}}(1+\ln^{2/3}\Lambda)a^*(\mathbb{E}I_L)^{-4/3}\ln l\big)\le(\mathbb{E}I_{L/l})^{4/3}\\
&\mbox{so that}\quad\Lambda\ln L\le\Lambda\ln(L/l)+\textstyle{\frac{4}{3}}(1+\ln^{2/3}\Lambda)a^*\ln l+O(\Lambda l^{-1}\ln L).
\end{align*}
Since the last inequality may be rewritten as $\Lambda\big(1-o(1)\big)\le\textstyle{\frac{4}{3}}(1+\ln^{2/3}\Lambda)a^*$, it implies a universal estimate on $\Lambda$, hence contradicting \eqref{m192}.
\qed

\medskip

{\sc Proof of Theorem \ref{T:5}. }
The main ingredient is the following consequence of Lemma~\ref{L:34}
\begin{align}\label{c:87}
\E I_{ l L}-\E I_L\lesssim(\ln l)\ln^{-1/4}L\quad\mbox{provided}~\ln l\gg\ln\ln L
\end{align}
Indeed, the assumption on $l$ implies the existence of $L=L_0<\ldots<L_N=lL$ with $L_n=l_nL_{n-1}$ and $(L_n,l_n)$ satisfying \eqref{m184}, so that from Lemma~\ref{L:34} and the monotonicity $\mathbb{E}I_{L_n}\ge\mathbb{E}I_L$, we have
\begin{align*}
\mathbb{E}I_{L_n}-\mathbb{E}I_{L_{n-1}}\lesssim(\mathbb{E}I_{L_n})^{-1/3}\ln l_n\le(\mathbb{E}I_L)^{-1/3}\ln l_n\lesssim(\ln l_n)/\ln^{1/4}L. 
\end{align*}
Summing over $n$, we deduce \eqref{c:87}.

\medskip

\emph{Proof of \eqref{c88}.}
Let $\nu=\ln l\gg\ln\ln L$. We can find $\{B_L(z_n)\}_{n=1}^N\subset B_{lL}$  disjoint balls of radius $L$ with $N\sim l^2$, so that
\begin{align*}
\mbox{let }	I_{L,n}:=\sup_{\gamma\subset z_n+\Gamma_{L,1}}\frac{W(\gamma)}{P(\gamma)},\mbox{ then $\{I_{L,n}\}_{n=1}^{N}$ are i.~i.~d.~with $I_{L,n}=_{\text{law}}I_L$}.
\end{align*}
Since $I_{lL}\ge\max_{n=1}^NI_{L,n}$, we deduce
\begin{align*}
\E\max_{n=1}^N(I_{L,n}-\E I_L)\le\mathbb{E}I_{lL}-\mathbb{E}I_L\overset{\eqref{c:87}}
{\lesssim}(\ln l)\ln^{-1/4}L=\nu\ln^{-1/4}L.
\end{align*}
In order to conclude, we need the following fact, which we prove below: If $\{X_n\}_{i=n}^N$ be i.~i.~d.~copies of a centered random variable $X$, then for  $N\gg1$
\begin{align}\label{m200}
\mathbb{P}(X\ge 2\mathbb{E}\max_{n=1}^NX_n)\le N^{-1}.
\end{align}

\medskip

Applying \eqref{m200} with $X_n:=I_{L,n}-\mathbb{E}I_L$, we deduce \eqref{c88}:
\begin{align*}
&\mathbb{P}(I_L-\mathbb{E}I_L\ge C\nu\ln^{-1/4}L)\\&\le\mathbb{P}\big(I_L-\mathbb{E}I_L\ge C\mathbb{E}\max_{n=1}^N(I_{L,n}-\mathbb{E}I_L)\big)\le N^{-1}\le e^{-c\nu}.
\end{align*}
\medskip

Here comes the proof of \eqref{m200}. We set $M_N:=\max_{n=1}^N X_n$, and observe that 
\begin{align*}
\E M_N&=\E \max\{M_{N-1},X_N \}\\
&\geq \E M_{N-1}I(M_{N-1}\geq 0 )+\E X_N I(M_{N-1}<0)=\E (M_{N-1})_+.
\end{align*}
Thus for any $\nu>0$, by Markov's inequality,
\begin{align*}
\nu^{-1}\E M_{N}&\geq\nu^{-1}\E (M_{N-1})_+\geq\PP(M_{N-1}\geq\nu)\\&=1-\PP(M_{N-1}<\nu)=1-(1-\PP(X\geq\nu))^{N-1}
\end{align*}
Now we take $\nu=2\E M_N$, then we obtain
\begin{align*}
1-(1-\PP(X\geq 2\E M_N))^{N-1}\leq \textstyle{}\frac{1}{2}.
\end{align*}
Consequently, \eqref{m200} must hold for $N\gg 1$. Indeed, if not, we obtain a contradiction from the above inequality:
\begin{align*}
\textstyle{\frac{1}{2}}\geq 1-(1-N^{-1})^{N-1}\to 1-e^{-1}\mbox{ as }N\to\infty.
\end{align*}

\medskip

\emph{Proof of \eqref{c89}.} 
Let $\ln\ln L\ll\nu=\ln l\ll\ln L$. Since this implies $\ln L\approx\ln(L/l)$, by replacing $L$ with $L/l$ in \eqref{c:87}, we have
\begin{align}
\mathbb{E}I_L-\mathbb{E}I_{L/l}\le C_0(\ln l)/\ln^{1/4}L.\label{c95}
\end{align}

\medskip

We start with the weaker statement
\begin{align}
\PP(I_L-\E I_L\leq-\nu\ln^{-1/4}L)&\leq e^{-1}\quad\mbox{provided}\quad\nu\gg\ln\ln L\label{c96}.
\end{align}
Indeed, by the right tail bound in \eqref{c88} and the fact that $I_L-\mathbb{E}I_L$ is centered,
\begin{align*}
\E (I_L-\E I_L)_-=\E (I_L-\E I_L)_+\lesssim (\ln\ln L)\ln^{-1/4}L,
\end{align*}
so that \eqref{c96} follows from Markov inequality. 

\medskip

Let us conclude the proof. As above, 
we can find $\{B_{L/l}(z_n) \}_{n=1}^{N}\subset B_L$ disjoint balls with $N\sim l^2$ such that 
\begin{align*}
\mbox{let }I_{L/l,n}:=\sup_{\gamma\subset z_n+\Gamma_{L/l,1}}\frac{W(\gamma)}{P(\gamma)},\mbox{ then $\{I_{L/l,n}\}_{n=1}^{N}$ are i.~i.~d.~with $I_{L/l,n}=_{\text{law}}I_{L/l}$}.
\end{align*}
Finally, since again $I_L\ge\max_{n=1}^NI_{L/l,n}$, using \eqref{c95} and \eqref{c96} (with $L$ replaced by $L/l$),
\begin{align*}
&\PP(I_L-\E I_L\leq-2C_0\nu\ln^{-1/4}L)\\&\leq \PP(\max_{n=1}^N I_{L/l,n}-\E I_{L/l}\leq-C_0\nu\ln^{-1/4}L)\\&=
\mathbb{P}(I_{L/l}-\E I_{L/l}\leq-C_0\nu\ln^{-1/4}L)^N\le e^{-N}\le e^{-e^{c\nu}}.
\end{align*}
\qed

\medskip

{\sc Proof of Lemma~\ref{L:28}.} \emph{Proof of (i).} Let $t^{ent}=t_0,t_1,t_2\in\bar{\mathbb{S}}_{in}$ be the first three vertices of $\gamma^*|_{\mathbb{S}_{in}}$ starting from $t^{ent}$ and let us pick the closest point $z\in\partial Q_L$ to $\gamma^*(t_2)$ such that $z$ is on the same side of $\partial Q_L$ of $\gamma^*(t^{ent})$ and $|z-\gamma^*(t_2)|\ge1$. A competitor is then given by replacing $\gamma^*|_{[t^{ent},t_2]}$ with the segment $[z,\gamma^*(t_2)]$ (and appropriately closing $\gamma^*$ along the boundary $\partial Q_L$). As in Lemma~\ref{L:12}, outside of an event of probability $\leq L^{-1}$,
\begin{align*}
2\le\len(\gamma^*|_{[t^{ent},t_2]})\le(1+\eta)|z-\gamma^*(t_2)|\quad\mbox{where}\quad\eta\lesssim\ln^{-1/4}L,
\end{align*}
which implies that $|z-\gamma^*(t_2)|>1$, namely $z$ is the projection of $\gamma^*(t_2)$ onto $\partial Q_L$, and by simple geometric considerations the conclusion on the normals in $(i)$. The same reasoning holds for $t^{ext}$.

\medskip

\emph{Proof of $(ii)$.} Assume by contradiction that $|\gamma^*(t^{ent})-z|\ll l$ for some vertex $z$ of $Q_L$. Let $S_1,S_2$ be the two sides of $\partial Q_L$ with intersection $z$, with $\gamma^*(t^{ent})\in S_1$. Consider the branch $\gamma^*_{t,l}$ with $t=t^{ent}$ and exit time $t^{out}_{t,l}$. Since this branch is almost flat and intersects $S_1$ almost perpendicularly, then $d(\gamma^*(t^{out}_{t,l}),S_2)\ll l$. As in $(i)$, let $z_2\in S_2$ be the closest point to $\gamma^*(t_{t,l}^{out})$ with the property that $|\gamma^*(t_{t,l}^{out})-z_2|\ge 1$. A competitor would then be given by replacing $\gamma^*_{t,l}$ with the segment $[\gamma^*(t^{out}_{t,l}),z_2]$. But this would lead to a contradiction since, with the same reasoning as in Lemma~\ref{L:12},
\begin{align*}
l-2\le\len(\gamma^*_{t,l})\le(1+\eta)|\gamma^*(t^{out}_{t,l})-z_2|\ll l.
\end{align*}

\medskip

\emph{Proof of $(iii)$.} Assume again by contradiction that $|\gamma^*(t^{ent})-\gamma^*(t^{ext})|\ll l$. A competitor is constructed by replacing the two branches $\gamma^*_{t^{ent},l},\gamma^*_{t^{ext},l}$ with a segment (or a small-scale modification of it) connecting the two exiting points $t^{out}_{t^{ent},l}$ and $t^{out}_{t^{ext},l}$. The conclusion follows as above.

\medskip

\emph{Proof of $(iv)$.} The strategy is the same as in $(ii)$ and $(iii)$, hence we just describe the construction of the competitor in the hypothesis that $\gamma^*_{t,l}$ does not intersect the boundary $\partial Q_L$. Let $S$ be the side of $Q_L$ such that $d(\gamma^*(t),S)=d(\gamma^*(t),Q_L)$ and pick $z_{in}\in S$ which minimizes $|z_{in}-\gamma^*(t^{in}_{t,l})|$ under the constraint $|z_{in}-\gamma^*(t^{in}_{t,l})|\ge 1$. Let $z_{out}$ be defined in the same way for $\gamma^*(t^{out}_{t,l})$. Then the competitor is given by replacing the branch $\gamma^*_{t,l}$ with the segments $[\gamma^*(t^{in}_{t,l}),z_{in}]$ and $[z_{out},\gamma^*(t^{out}_{t,l})]$.
\qed


\section{Appendix}

\subsection{Auxiliary lemmas}

\begin{lemma}[{\cite[Lemmas 5 and 14]{opw}}]\label{app:4}
For $s\in[1,\infty)$ and positive random variables $X, \{X_n\}_{n\ge1}$
\begin{align}
\label{c:85}& \mbox{if}\quad\forall \nu\geq \nu_0\:\:\ln\PP(X\geq \nu)\leq-\nu^s\quad\mbox{then}\quad\|X\|_s\lesssim \nu_0+1;\\
&\ln \PP(X\geq \nu)\lesssim -(\nu/\|X\|_s)^s\quad\text{for}~\nu\gg \|X\|_s;\label{c:86}
\end{align}
Moreover there exists $\bar X$ with $\|\bar X\|_s\le1$ such that
\begin{align}
&X_n/\|X_n\|_s\lesssim\ln^{1/s}n+\bar X\quad\mbox{for all $n\ge1$}\nonumber\\
&\mbox{which implies}\quad\|\max_{n=1}^NX_n\|_s\lesssim(\max_{n=1}^N\|X_n\|_s)\ln^{1/s}N.\label{m83}
\end{align}
\end{lemma}

\begin{lemma}\label{app:1}
Let $\{W(x)\}_{x\in B_1}$ be a centered Gaussian process such that 
\begin{align}
W(0)=0\quad\mbox{and}\quad\E^{1/2}(W(x)-W(y))^2\leq |x-y|^{\alpha},\label{c:14}
\end{align}
for all $x,y\in B_1\subset\mathbb{R}^d$ and some $\alpha\in(0,1]$. Then
\begin{align*}
\|\sup_{x\in B_1}W(x)\|_2\lesssim_{d,\alpha}1.
\end{align*}
\end{lemma}

{\sc Proof.} By \eqref{c:14} and Borell's inequality, we have 
\begin{align*}
\|\sup_{x\in B_1}W(x)-\E\sup_{x\in B_1}W(x)\|_2\lesssim 1,
\end{align*}
it suffices to show that $\E\sup_{x\in B_1}W(x)\lesssim_{d,\alpha}1$. Note that by \eqref{c:14} and Gaussianity, 
\begin{align*}
\E^{1/p}|W(x)-W(y)|^p\lesssim_p |x-y|^{\alpha}\text{ for all }x,y\in B_1\text{ and for }p\geq 1.
\end{align*}
Therefore, for any $\alpha'<\alpha$, we control the $\dot{W}^{\alpha',p}$ norm of $W$ in expectation,
\begin{align*}
&\E \int_{B_1}\int_{B_1}\left(\frac{|W(x)-W(y)|}{|x-y|^{\alpha'}} \right)^p\frac{1}{|x-y|^d}dxdy\\&\lesssim_p \int_{B_1}\int_{B_1}\frac{1}{|x-y|^{p(\alpha'-\alpha)+d}}dxdy\lesssim _{p,d,\alpha',\alpha}1.
\end{align*}
Let us fix $\alpha'=\alpha/2$ and $p>2d/\alpha$, then by Sobolev embedding, we obtain
\begin{align*}
\E [W]_{C^{\alpha/2-d/p}}^p\lesssim_{d,\alpha} 1.
\end{align*}
Combined with $W(0)=0$, we see that $\E\sup_{x\in B_1}W(x)\lesssim_{d,\alpha}1$.
\qed

\begin{lemma}\label{L:37}
\begin{align*}
\mathbb{E}I_L<\infty
\end{align*}
\end{lemma}

{\sc Proof.} Given a curve $\gamma\in\Gamma_{L,1}$ with $N$ vertices $z_1,\ldots,z_N$, one may decompose it into $N-2$ triangular curves $\gamma_n$ with vertices $(z_1,z_n,z_{n+1})$ for $n=2,\ldots,N-1$. Since $\len(\gamma)\ge N$,
\begin{align*}
\frac{W(\gamma)}{\len(\gamma)}\le\frac{1}{N}\sum_{n=2}^{N-1}W(\gamma_n)\le\sup_{\text{triangles}~T\subset B_L}|W(T)|.
\end{align*}
The set of triangles $T$ contained in $B_L$, parameterized by the coordinates of their 3 vertices, is finite dimensional and compact, so that the finiteness of the r.~h.~s.~follows from Lemma~\ref{app:1}. 
\qed

\medskip

{\sc Proof of Lemma~\ref{T:7} (Lower bound)} Let us show that
\begin{align}\label{m162}
\mathbb{E}I_L\gtrsim\ln^{3/4}L.
\end{align}
To this end, we construct a competitor by taking the maximizer $h^*\in\mathcal{H}_{L,1}$ of the rescaled action $\epsilon W-D$ over the interval $[-L/2,L/2]
\times\{0\}\subset B_L$ (cf.~\eqref{m163}). Note that the scaling
\begin{align}\label{m133}
h=\epsilon^{2/3}\hat h\quad\mbox{implies}\quad\epsilon W-D=_{\text{law}}\epsilon^{4/3}(\hat W-D),
\end{align}
so that from Theorem~\ref{T:3} and the choice $\epsilon:=\ln^{-3/4}L$, we have
\begin{align*}
(\epsilon W-D)(h^*)/L\approx a^*\quad\mbox{with overwhelming probability.}
\end{align*}
Consider the closed curve $\gamma$ given by the concatenation $[-L/2,L/2]\times\{0\}$ with the graph of $h^*$ (traveled backward). From \eqref{m65} and \eqref{m107}, one has
\begin{align}\label{m76}
W(\gamma)=W(h^*).
\end{align}
The uniform  bound of $|h^*|$ in \eqref{m75} and the fact that $\epsilon\ll1$ ensure that
\begin{align*}
\mbox{the image of $\gamma$ is contained in $B_L$ with overwhelming probability.}
\end{align*}
Together with the small-scale cutoff of $h^*$, the above implies that $\gamma\in\Gamma_{L,1}$. The identity \eqref{m76} is complemented by the estimate of the length
\begin{align*}
\len(\gamma)\le L+\int_{-L/2}^{L/2}dx\sqrt{1+(\frac{dh^*}{dx})^2}\le 2L+D(h^*),
\end{align*}
so that the ratio is lower bounded by
\begin{align}\label{m81}
\frac{W(\gamma)}{\len(\gamma)}\ge\frac{W(h^*)}{2L+D(h^*)}=\frac{1}{\epsilon}\frac{\big(\epsilon W(h^*)-D(h^*)\big)+D(h^*)}{2L+D(h^*)}\gtrsim\ln^{3/4}L.
\end{align}
with high probability. Taking expectations yields \eqref{m162}.\qed

\begin{lemma}\label{app:3}
Let $ \gamma :\mathbb{S}\to \mathbb{R}^2$  be a simple, piecewise-linear, oriented, closed curve and let $ B\subset \mathbb{R}^2 $ be an arbitrary ball.  Define 
\begin{align*}
&\mathcal{I}:=\{z\in\partial B\mid \mbox{$\gamma$ enters $B$ at $z$}\}\quad\mbox{and}\\&\mathcal{O}:=\{z\in\partial B\mid \mbox{$\gamma$ exits $B$ at $z$}\}.
\end{align*}
Then $\mathcal{I}$ and $\mathcal{O}$ have the same cardinality $k\in\mathbb{N}$ and their points alternate along $\partial B$.
\end{lemma}

{\sc Proof of Lemma~\ref{app:3}.} Since $\gamma$ is piecewise linear and $\mathbb{S}$ is compact, then $\gamma^{-1}(\mathcal{I}\cup\mathcal{O})\subset\mathbb{S}$ consists of finitely many, isolated points.
By Jordan's theorem, $\mathbb{R}^2\backslash\{\mbox{image of $\gamma$}\}$ is made of two open connected components: a bounded one $U_1$ and an unbounded one $U_2$. Up to choosing the right orientation of $\partial B$, every intersection point $z\in\mathcal{I}\cup\mathcal{O}$ satisfies
\begin{align*}
&\mbox{$\gamma$ enters $B$ at $z$}\quad\Longleftrightarrow\quad\mbox{$\partial B$ enters $U_1$ at $z$}\\
&\mbox{$\gamma$ exits $B$ at $z$}\quad\Longleftrightarrow\quad\mbox{$\partial B$ exits $U_1$ at $z$}.
\end{align*}
Hence the claim is equivalent to proving that, along $\partial B$, the points where $\partial B$ enters and exits $U_1$ alternate. But this is obvious.
\qed

\subsection{Proof of Theorem~\ref{T:3}}

The main result in \cite{op} is a statement analogous to Theorem~\ref{T:3} where $\mathcal{H}_{l,1}$ is replaced by the smaller class
\begin{align*}
\tilde{\mathcal{H}}_{l,1}:=\{&h:[0,l]\to\mathbb{R}~\mbox{piecewise linear on intervals $[n-1,n]$}\\&\mbox{for $n=1,\ldots,l$ with}~h(0)=h(l)=0\}.
\end{align*}

\begin{theorem}[{\cite[Theorem 1]{op}}]\label{T:4}
\begin{align*}
\|\max_{h\in\tilde{\mathcal{H}}_{l,1}}(W-D)(h)/l-a^*\ln l\|_{3/2}\lesssim1.
\end{align*}
\end{theorem}

We will obtain Theorem~\ref{T:3} by approximation from this result. Note that the inclusion $\tilde{\mathcal{H}}_{l,1}\subset\mathcal{H}_{l,1}$ directly implies
\begin{align*}
\max_{h\in\tilde{\mathcal{H}}_{l,1}}(W-D)(h)\le\max_{h\in\mathcal{H}_{l,1}}(W-D)(h),
\end{align*}
so that we just need the other inequality. To this end, let us introduce the projection $\tilde\Pi:\mathcal{H}_{l,2}\to\tilde{\mathcal{H}}_{l,1}$ -- the change from 1 to 2 in the microscale will be convenient later -- 
\begin{align*}
\tilde\Pi(h)~\mbox{is the piecewise linear interpolation of $h$ at points $n=0,1,\ldots,l$.}
\end{align*}
Let $h^*$ be a maximizer in $\mathcal{H}_{l,2}$. Note that $D(\tilde\Pi(h))\le D(h)$, so that 
\begin{align}\label{m116}
(W-D)(\tilde\Pi(h^*))\ge(W-D)(h^*)+\big(W(\tilde\Pi(h^*))-W(h^*)\big).
\end{align}
Since $\tilde\Pi(h^*)$ is a competitor for the problem in $\tilde{\mathcal{H}}_{l,1}$, it implies
\begin{align*}
\max_{\tilde h\in\tilde{\mathcal{H}}_{l,1}}(W-D)(\tilde h)\ge\max_{h\in\mathcal{H}_{l,2}}(W-D)(h)+\big(W(\tilde\Pi(h^*))-W(h^*)\big).
\end{align*}
Together with Theorem~\ref{T:4}, we have deduced the following.

\begin{lemma}\label{L:22}
\begin{align*}
\|\max_{h\in\mathcal{H}_{l,2}}(W-D)(h)/l-a^*\ln l\|_{3/2}\lesssim1+\big\|W(\tilde\Pi(h^*))-W(h^*)\big\|_{3/2}/l.
\end{align*}
\end{lemma}

Let us describe the strategy to control the r.~h.~s.~. For a fixed $h\in\mathcal{H}_{l,2}$, let us decompose the error as $l$ summands
\begin{align*}
|W(\tilde\Pi(h))-W(h)|\le\sum_{n=1}^l\big|\int_{n-1}^ndx\int_{h(x)}^{\tilde\Pi(h)(x)}dy~\xi\big|.
\end{align*}
Let us pick $n$ such that $h$ and $\tilde\Pi(h)$ do not coincide in $[n,n+1]$. Since
\begin{align*}
h\in\mathcal{H}_{l,2}\quad\mbox{then}\quad\mbox{$h$ and $\tilde\Pi(h)$ coincide in $[n-1,n]\cup [n+1,n+2]$.}
\end{align*}
We notice that the red segment in Figure~\ref{f7} has length equal to the difference of two consecutive derivatives $\big|(h(n)-h(n-1))-(h(n+1)-h(n))\big|$ and therefore the $n$-th term on the r.~h.~s.~above is the integral of $\xi$ over a triangle of area $\le{\textstyle{\frac{1}{2}}}\big|(h(n)-h(n-1))-(h(n+1)-h(n))\big|$ so that
\begin{align}\label{m110}
\big\|\int_{n-1}^ndx\int_{h(x)}^{\tilde\Pi(h)(x)}dy~\xi\big\|_2\lesssim\big|h(n)-{\textstyle\frac{1}{2}}\big(h(n-1)+h(n+1)\big)\big|^{1/2}.
\end{align}
%
\begin{figure}[!ht]
\centering
\resizebox{0.6\textwidth}{!}{%
\begin{tikzpicture}[scale=0.6]
\tikzstyle{every node}=[font=\fontsize{18.2pt}{23.7pt}\selectfont]
\draw [dashed] (-2.5,9.875) -- (-2.5,2.5);
\draw [dashed] (7.5,9.875) -- (7.5,2.5);
\draw [dashed] (17.5,9.875) -- (17.5,2.5);
\draw [short] (13.75,8.75) -- (-3,3.125);
\draw [short] (13.75,8.75) -- (18.75,3.75);
\draw [dashed] (13.75,8.75) -- (17.5,10);
\draw [line width=1pt, short] (-2.5,3.25) -- (7.5,6.625);
\draw [line width=1pt, short] (7.5,6.625) -- (17.5,5);
\draw [line width=1pt, short] (17.5,5) -- (18.75,3.75);
\draw [ color={rgb,255:red,255; green,0; blue,0}, draw opacity=1, line width=1pt, short] (17.5,10) -- (17.5,5);
\node [font=\fontsize{18.2pt}{23.7pt}\selectfont, fill={rgb,255:red,255; green,255; blue,255}, fill opacity=1, text opacity=1, inner xsep=0.080cm, inner ysep=0.085cm, rounded corners=0.020cm] at (-2.5,2.75) {$n-1$};
\node [font=\fontsize{18.2pt}{23.7pt}\selectfont, fill={rgb,255:red,255; green,255; blue,255}, fill opacity=1, text opacity=1, inner xsep=0.080cm, inner ysep=0.085cm, rounded corners=0.020cm] at (7.5,2.75) {$n$};
\node [font=\fontsize{18.2pt}{23.7pt}\selectfont, fill={rgb,255:red,255; green,255; blue,255}, fill opacity=1, text opacity=1, inner xsep=0.080cm, inner ysep=0.085cm, rounded corners=0.020cm] at (17.5,2.75) {$n+1$};
\node [font=\fontsize{18.2pt}{23.7pt}\selectfont, fill={rgb,255:red,255; green,255; blue,255}, fill opacity=1, text opacity=1, inner xsep=0.080cm, inner ysep=0.085cm, rounded corners=0.020cm] at (12.75,5.375) {$\tilde\Pi(h)$};
\node [font=\fontsize{18.2pt}{23.7pt}\selectfont, fill={rgb,255:red,255; green,255; blue,255}, fill opacity=1, text opacity=1, inner xsep=0.080cm, inner ysep=0.085cm, rounded corners=0.020cm] at (12.125,8.25) {$h$};
\end{tikzpicture}
}%
\caption{Difference between $h$ and $\tilde\Pi(h)$.}
\label{f7}
\end{figure}
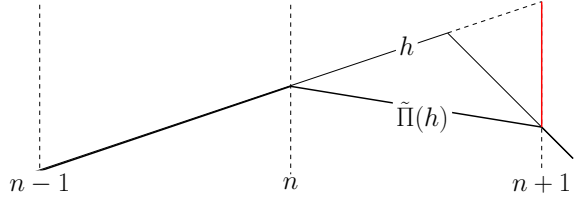

\medskip

We would like to show a similar bound when a deterministic $h$ is replaced by the random $h^*$. As in \cite{op} we introduce the net in the configuration space $\mathcal{H}_{l,2}$
\begin{align}\label{m122}
\mathcal{N}:=\{h\in\tilde{\mathcal{H}}_{l,1}~|~h(n)\in\mathbb{Z}~\mbox{for all}~n=0,\ldots,l\},
\end{align}
which comes with the projection operator $\Pi:\mathcal{H}_{l,2}\to\mathcal{N}$ characterized by
\begin{align}\label{m123}
h(n)\in[\Pi(h)(n)-\textstyle{\frac{1}{2}},\Pi(h)(n)+\textstyle{\frac{1}{2}})\quad\mbox{so that}\quad|h(n)-\Pi(h)(n)|\le{\textstyle{\frac{1}{2}}}.
\end{align}
We denote elements of the net $\mathcal{N}$ by $\bar h$.

\medskip

The first step is to fix some $\bar h\in\mathcal{N}$ and upgrade \eqref{m110} to a uniform estimate over all $h$ with the same projection $\Pi(h)=\bar h$. Since the triangular area defining the random variable on the l.~h.~s.~of \eqref{m110} depends only on the value $h$ at the four points $n-1,n,n+1,n+2$, this amounts to control the supremum of a Gaussian field over a compact set in a 4-dimensional space. Applying Lemma~\ref{app:1}, one upgrades \eqref{m110} to the following: For every $n$ and $\bar h$, 
\begin{align}
&E_n(\bar h):=\sup_{h:\Pi(h)=\bar h}\int_{n-1}^ndx\int_{h(x)}^{\tilde\Pi(h) (x)}dy~\xi\nonumber\\&\mbox{satisfies}\quad\|E_n(\bar h)\|_2\lesssim\big|\bar h(n)-{\textstyle\frac{1}{2}}\big(\bar h(n-1)+\bar h(n+1)\big)\big|^{1/2}+1.\label{m111}
\end{align}
With this notation, the r.~h.~s.~of Lemma~\ref{L:22} is controlled from above by
\begin{align}\label{m114}
|W(\tilde\Pi(h^*))-W(h^*)|/l\le l^{-1}\sum_{n=1}^lE_n(\Pi(h^*)).
\end{align}

\medskip

We recall the scale-by-scale decomposition for $h\in\tilde{\mathcal{H}}_{l,1}$. For $1\le \rho\in2^{-\mathbb{N}}l$
\begin{align*}
&h_{\ge\rho}:=\mbox{piecewise linear interpolation of $h$ at $n\rho$ for $n=0,\ldots,l/\rho$}\\
&\mbox{and}\quad h_\rho:=h_{\ge \rho}-h_{\ge 2\rho}\quad\mbox{so that}\quad h=\sum_{1\le \rho\le l} h_\rho.
\end{align*}
This sum is orthogonal with respect to the Dirichlet energy
\begin{align*}
D(h)=\sum_\rho D(h_\rho).
\end{align*}
Moreover, we make a connection with the estimate \eqref{m111} by noting that for odd $n$ the higher component $h_1$ can be written as
\begin{align}
&h_1(n)=h(n)-{\textstyle{\frac{1}{2}}}\big(h(n-1)+h(n+1)\big)\nonumber\\
&\mbox{so that}\quad D(h_1)/l=l^{-1}\sum_{n~\text{odd}}(h(n)-{\textstyle{\frac{1}{2}}}\big(h(n-1)+h(n+1))\big)^2.\label{m112}
\end{align}
Note that the estimate in \eqref{m111} could have been equivalently written in terms of the points $n,n+1,n+2$. Making this latter choice whenever $n$ is even and using \eqref{m112}, we obtain
\begin{align}\label{m115}
\big(l^{-1}\sum_{n=1}^l\|E_n(\bar h)\|_2^4\big)^{1/4}\lesssim(D(\bar h_1)/l)^{1/4}+1.
\end{align}
This motivates the following variant of Lemma 4 in \cite{op}.

\begin{lemma}
Suppose that for every deterministic $\bar h\in\mathcal{N}$ we are given a
family $\{E_n(\bar h)\}_{n=1,\ldots,l}$ of independent random variables.
Suppose that for every threshold $\nu\ge 2,\nu\in\mathbb{N}$ there exists a
subset $\mathcal{N}_\nu\subset\mathcal{N}$ with
\begin{align}
&\ln\#\mathcal{N}_\nu\le C_0l\ln\nu\quad\mbox{and such that},\label{m113}\\&
\big(l^{-1}\sum_{n=1}^l\|E_n(\bar h)\|_2^4\big)^{1/4}\le\nu^{1/4}\quad\mbox{for every $\bar h\in\mathcal{N}_\nu$.}\label{m119}
\end{align}
Then there exists $\bar E$ with $\|\bar E\|_{3/2}\le1$ such that for every $\nu$ and $\bar h\in\mathcal{N}_\nu$
\begin{align*}
l^{-1}\sum_{n=1}^lE_n(\bar h)\lesssim\nu^{1/4}(\ln^{1/2}\nu+l^{-1/2}\bar E).
\end{align*}
\end{lemma}

{\sc Proof.} Applying H\"older's and Jensen's inequality and \eqref{m119}, one obtains
\begin{align*}
l^{-1}\sum_{n=1}^{l}E_n(\bar h)&\le\big(l^{-1}\sum_{n=1}^l\|E_n(\bar h)\|_2^4\big)^{1/4}\big(l^{-1}\sum_{n=1}^l(\frac{E_n(\bar h)}{\|E_n(\bar h)\|_2})^{4/3}\big)^{3/4}\\&
\le\nu^{1/4}\big(l^{-1}\sum_{n=1}^l(\frac{E_n(\bar h)}{\|E_n(\bar h)\|_2})^2\big)^{1/2}.
\end{align*}
Together with a union bound, Chebyshev inequality and independence, we obtain that
\begin{align*}
&\mathbb{P}(\max_{\bar h\in\mathcal{N}_\nu}l^{-1}\sum_{n=1}^l E_n(\bar h)\ge\nu^{1/4}\mu)\\&\le\#\mathcal{N}_\nu\mathbb{P}(l^{-1}\sum_{n=1}^l(\frac{E_n(\bar h)}{\|E_n(\bar h)\|_2})^2\ge\mu^2)\le e^{-l\mu^2+l+C_0l\ln\nu}.
\end{align*}
Changing the variables to $\mu=C\ln\nu+l^{-1/2}\tilde\mu$ for some $C\gg C_0+1$ and using \eqref{c:85}, this can be equivalently written as: For all $\bar h\in\mathcal{N}_\nu$,
\begin{align*}
l^{-1}\sum_{n=1}^l E_n(\bar h)\lesssim\nu^{1/4}(\ln^{1/2}\nu+l^{-1/2}\bar E_\nu)\quad\mbox{where}\quad\|\bar E_\nu\|_{2}\lesssim1.
\end{align*}
We conclude by applying \eqref{m83} to replace $\bar E_\nu$ with $\bar E$ independent of $\nu$.
\qed

\medskip

The subsets $\mathcal{N}_\nu$ in the previous lemma can be defined by
\begin{align*}
\mathcal{N}_\nu:=\{\bar h\in\mathcal{N}~|~D(\bar h_\rho)/l\le \rho^2\nu~\mbox{for all dyadic scales}~1\le \rho\le l\}.
\end{align*}
They satisfy the requirement \eqref{m113} by applying \cite[Lemma 6]{op}. Since
\begin{align*}
\bar h\in\mathcal{N}_\nu\quad\mbox{for some}\quad\nu\lesssim1+\max_\rho D(h_\rho)/l,
\end{align*}
applying the previous lemma, we control the r.~h.~s.~of \eqref{m114} as
\begin{align*}
&l^{-1}\sum_{n=1}^lE_n(\Pi(h^*))\\&\lesssim(1+\max_\rho D(h^*_\rho)/l)^{1/4}\big(\ln(e+\max_\rho D(h^*_\rho)/l)+l^{-1/2}\bar E).
\end{align*}
Let us take the norm $\|\cdot\|_{3/2}$ on both sides. Using H\"older's inequality for Orlicz norms $\|X^{1/4}Y\|_{3/2}\lesssim\|X\|^{1/4}_{3/2}\|Y\|_{2}$ to control the product with $\bar E$, we deduce the following

\begin{lemma}\label{L:24}
\begin{align*}
&\big\|W(\tilde\Pi(h^*))-W(h^*)\big\|_{3/2}/l\\&\lesssim\big(1+\|\max_\rho D(h_\rho^*)\|_{3/2}/l\big)\ln\big(e+\|\max_{\rho}D(h^*_\rho)\|_{3/2}/l\big).
\end{align*}
\end{lemma}

We are led to study the regularity of $h^*$ and we would like to prove that
\begin{align*}
\|\max_\rho D(h^*_\rho)/l\|_{3/2}\lesssim1.
\end{align*}
Note that the l.~h.~s.~does not change if we write $\tilde\Pi(h^*)$ in place of $h^*$. The above regularity has been proven in \cite{op} when $h^*$ is replaced by the maximizer $\tilde h^*$ of $W-D$ in the class $\tilde{\mathcal{H}}_{l,1}$. We will take the following route:
\begin{itemize}
\item We show that $\tilde\Pi(h^*)$ is in fact an almost-maximizer of $W-D$ in $\mathcal{H}_{l,1}$;
\item We apply the regularity for almost maximizers in \cite[Lemma 14]{op}.
\end{itemize}

\medskip

Let us show the first point starting from \eqref{m116}. The gap between the optimal action in $\tilde{\mathcal{H}}_{l,1}$ and the one of $\tilde\Pi(h^*)$ is controlled by
\begin{align*}
\max_{h\in\tilde{\mathcal{H}}_{l,1}}&(W-D)(h)-(W-D)(\tilde\Pi(h^*))\\
&\le\max_{h\in\tilde{\mathcal{H}}_{l,1}}(W-D)(h)-\max_{h\in\mathcal{H}_{l,2}}(W-D)(h)+\big(W(h^*)-W(\tilde\Pi(h^*))\big).
\end{align*}
On the r.~h.~s.~we can replace $\mathcal{H}_{l,2}$ by $\tilde{\mathcal{H}}_{l,2}$ by monotonicity. Adding and subtracting $a^*\ln l$ and applying Theorem~\ref{T:4} yields the following result, which we interpret as a almost-optimality statement for $\tilde\Pi(h^*)$ in $\tilde{\mathcal{H}}_{l,1}$.

\begin{lemma}\label{L:25}
\begin{align*}
&\|\max_{h\in\tilde{\mathcal{H}}_{l,1}}(W-D)(h)/l-(W-D)(\tilde\Pi(h^*))/l\|_{3/2}\\&\lesssim 1+\big\|W(h^*)-W(\tilde\Pi(h^*))\big\|_{3/2}/l.
\end{align*}

\end{lemma}

The last ingredient is a regularity result for almost maximizers, proved in \cite[Lemma 14]{op}: There exists a random variable $\bar D$ with $\|\bar D\|_{3/2}\le1$ such that for every $\tilde h\in\tilde{\mathcal{H}}_{l,1}$,
\begin{align*}
&\max_\rho D(\tilde h_\rho)/l\lesssim\bar D+\Lambda(\tilde h)\\&\mbox{where}\quad\Lambda(h):=\max_{\tilde h'\in\tilde{\mathcal{H}}_{l,1}}(W-D)(\tilde h')-(W-D)(\tilde h)
\end{align*}
is the energy deficit. Let us now prove Theorem~\ref{T:3}. Using Lemma~\ref{L:25}, we deduce that
\begin{align*}
\|\max_\rho D(h^*_\rho)/l\|_{3/2}=\|D(\tilde\Pi(h^*))/l\|_{3/2}\lesssim1+\big\|W(h^*)-W(\tilde\Pi(h^*))\big\|_{3/2}/l.
\end{align*}
Substituting this result in Lemma~\ref{L:24} and using that the function $x-Cx^{1/4}\ln(e+x)$ is coercive, we get
\begin{align*}
\|W(h^*)-W(\tilde\Pi(h^*))\|_{3/2}\lesssim1\quad\mbox{and thus}\quad\|\max_\rho D(\tilde\Pi(h^*)_\rho)\|_{3/2}/l\lesssim1.
\end{align*}
The former estimate, together with Lemma~\ref{L:22}, gives the first part of Theorem~\ref{T:3}. By summing over the scales, the latter gives
\begin{align}\label{m117}
\|D(\tilde\Pi(h^*))/l\|_{3/2}\lesssim\ln l.
\end{align}
Moreover, together with the Poincaré inequality
\begin{align*}
\max_{x\in[0,l]}|h(x)|/l\le\sum_{\rho}\max_{x\in[0,l]}|h_\rho(x)|/l\overset{\text{Poincaré}}{\lesssim}\sum_{\rho}(l/\rho)^{-1}(D(h_\rho)/l)^{1/2}
\end{align*}
it also yields
\begin{align}\label{m118}
\|\max_{x\in[0,l]}|\tilde\Pi(h^*)(x)|/l\|_3\lesssim\sum_{\rho}(l/\rho)^{-1}\|D(\tilde\Pi(h^*)_\rho)/l\|_{3/2}^{1/2}\lesssim1.
\end{align}
Finally, since $h^*\in\mathcal{H}_{l,2}$ -- note the microscopic scale 2 -- with the same reasoning that led to \eqref{m110}, one has
\begin{align*}
&\max_{x\in[0,l]}|h^*(x)-\tilde\Pi(h^*)(x)|/l\lesssim (D(\tilde\Pi(h^*)_1)/l)^{1/2}\quad\mbox{and}\\&|D(h^*)-D(\tilde\Pi(h^*))|\lesssim D(\tilde\Pi(h^*)_1).
\end{align*}
Therefore, from \eqref{m117} and \eqref{m118} we obtain that the same bounds also hold for $h^*$, hence the second part of Theorem~\ref{T:3}.

\section{Acknowledgements}

The authors thank C. Wagner and A. Wachtel for valuable discussions. 


\end{document}